\documentclass[english]{extarticle}
\usepackage[T2A,T1]{fontenc}
\usepackage[latin9]{inputenc}
\usepackage{geometry}
\usepackage{mathrsfs}
\usepackage{mathtools}
\usepackage{amsmath}
\usepackage{amsthm}
\usepackage{amssymb}
\usepackage{stackrel}
\usepackage{setspace}

\makeatletter

\DeclareRobustCommand{\cyrtext}{%
  \fontencoding{T2A}\selectfont\def\encodingdefault{T2A}}
\DeclareRobustCommand{\textcyr}[1]{\leavevmode{\cyrtext #1}}

\theoremstyle{plain}
\newtheorem{thm}{\protect\theoremname}[section]
\theoremstyle{definition}
\newtheorem{defn}[thm]{\protect\definitionname}
\theoremstyle{remark}
\newtheorem{rem}[thm]{\protect\remarkname}
\theoremstyle{plain}
\newtheorem{lem}[thm]{\protect\lemmaname}
\ifx\proof\undefined
\newenvironment{proof}[1][\protect\proofname]{\par
	\normalfont\topsep6\p@\@plus6\p@\relax
	\trivlist
	\itemindent\parindent
	\item[\hskip\labelsep\scshape #1]\ignorespaces
}{%
	\endtrivlist\@endpefalse
}
\providecommand{\proofname}{Proof}
\fi
\theoremstyle{plain}
\newtheorem{prop}[thm]{\protect\propositionname}
\theoremstyle{remark}
\newtheorem{claim}[thm]{\protect\claimname}
\theoremstyle{plain}
\newtheorem{cor}[thm]{\protect\corollaryname}

\@ifundefined{date}{}{\date{}}
\makeatother

\usepackage{babel}
\providecommand{\claimname}{Claim}
\providecommand{\corollaryname}{Corollary}
\providecommand{\definitionname}{Definition}
\providecommand{\lemmaname}{Lemma}
\providecommand{\propositionname}{Proposition}
\providecommand{\remarkname}{Remark}
\providecommand{\theoremname}{Theorem}

\begin{document}
\title{\textbf{\Large{}THE MAGNETIC LIOUVILLE EQUATION AS A SEMI-CLASSICAL
LIMIT}}
\maketitle
\begin{center}
IMMANUEL BEN PORAT\footnote{Mathematical Institute, University of Oxford, Oxford OX2 6GG, UK.
Immanuel.BenPorat@maths.ox.ac.uk} 
\par\end{center}
\begin{abstract}
The Liouville equation with non-constant magnetic field is obtained
as a limit in the Planck constant $\hbar$ of the von Neumann equation
with the same magnetic field. The convergence is with respect to an
appropriate semi-classical pseudo distance, and consequently with
respect to the Monge-Kantorovich distance. Uniform estimates both
in $\epsilon$ and $\hbar$ are proved for the specific 2D case of
a magnetic vector potential of the form $\frac{1}{\epsilon}x^{\bot}$.
As an application, an observation inequality for the von Neumann equation
with a magnetic vector potential is obtained. These results are a
magnetic variant of the works {[}F. Golse and T. Paul. Arch. Ration.
Mech. Anal., \textbf{223} (1): 57--94 (2017){]} and {[}F. Golse and
T. Paul. Math. Models Methods Appl. Sci., World Scientific Publishing,
In press, \textbf{32} (5) (2022){]} respectively. 
\end{abstract}

\section{Introduction }

We derive rigorously the Liouville equation with non-constant magnetic
field from the von Neumann equation with the same magnetic field.
The Liouville equation with magnetic vector potential $A(x)$, electric
potential $V$ and initial data $f^{\mathrm{in}}$ reads 

\begin{equation}
\left\{ \begin{array}{c}
\partial_{t}f+\left\{ \frac{1}{2}\left|\xi+A(x)\right|^{2}+\frac{1}{2}\left|x\right|^{2}+V(x),f\right\} =0,\\
f|_{t=0}=f^{\mathrm{in}},
\end{array}\right.\label{Vlasov}
\end{equation}
where the unknown $f\coloneqq f(t,x,\xi)$ is a time dependent ($t\in[0,\tau]\subset[0,\infty)$)
probability density on $\mathbb{R}^{d}\times\mathbb{R}^{d}$, and
$\left\{ \cdot,\cdot\right\} $ are the Poisson brackets defined by 

\[
\left\{ f,g\right\} =\nabla_{\xi}f\cdot\nabla_{x}g-\nabla_{x}f\cdot\nabla_{\xi}g.
\]
The quadratic term $\frac{1}{2}\left|x\right|^{2}$ can be interpeted
as an confinement potential and its inclusion is a technicality, in
the sense that its utility is not apparent at the level of the formal
calculations but at the level of the spectral theory required in order
to put the calculations on rigours grounds. The von Neumann equation
with magnetic vector potential $A(x)$, electric potential $V$ and
initial data $R_{\hbar}^{\mathrm{in}}$ reads 

\begin{equation}
\left\{ \begin{array}{c}
\partial_{t}R_{\hbar}+\frac{i}{\hbar}\left[\frac{1}{2}\left|-i\hbar\nabla_{y}+A(y)\right|^{2}+\frac{1}{2}\left|y\right|^{2}+V(y),R_{\hbar}\right]=0,\\
R|_{t=0}=R_{\hbar}^{\mathrm{in}}.
\end{array}\right.\label{Hartree}
\end{equation}
Here $\hbar>0$ is the Planck constant, the unknown $R_{\hbar}$ is
a time dependent density operator on $\mathfrak{H}\coloneqq L^{2}(\mathbb{R}^{d})$
and $\left[T,S\right]$ designates the commutator of operators $T$
and $S$ defined by 

\[
\left[T,S\right]=TS-ST.
\]
By a density operator on a Hilbert space $\mathfrak{H}$ we mean a
trace class operator $R$ on $\mathfrak{H}$ such that 

\[
R=R^{\ast}\geq0,\ \mathrm{trace}_{\mathfrak{H}}(R)=1,
\]
and the set of all such operators is denoted $\mathcal{D}(\mathfrak{H})$. 

\medskip{}
The derivation of the Liouville equation from the von Neumann equation
in the absence of a magnetic field was investigated in {[}14{]}--
in particular see Theorem 2.5. In fact, the authors consider there
the slightly more involved case of the Vlasov and Hartree equations,
which are nonlinear versions of equations (\ref{Vlasov}) and (\ref{Hartree})
respectively. The authors in {[}14{]} successfully derive Vlasov
from Hartree (or Liouville from von Neumann for our purposes) with
respect to the Monge-Kantorovich distance of exponent 2 in the limit
as $\hbar\rightarrow0$. When no magnetic field is included and the
electric potential $V$ is sufficiently well behaved in terms of regularity
and decay, the existing results on quantum, semi-classical and classical
mean field limits are fairly satisfactory. The case of singular potentials
(such as the Coulomb potential) is more challenging: in {[}23{]}
the Vlasov equation is obtained as a semi-classical limit of the Hartree
equation, with Coulomb potential. However the method introduced in
{[}23{]} does not yield quantitative estimates and is obtained along
a sub-sequence of $\hbar$. In {[}2{]} it is shown that the Wigner
function of the solution of Hartree's equation is $L^{2}$-close to
its weak limit, that is, the solution of Vlasov, but this convergence,
which is with respect to a semi-classical analogue of the Monge-Kantorovich
distance, is obtained provided the electric potential is sufficiently
regular and decaying and the Sobolev norm of the initial datas satisfies
appropriate growth rates with respect to $\hbar$. More recently,
the problem of establishing a quantiative rate of convergence for
the semi-classical limit with respect to in the presence of singular
interactions (Coulomb singularity included) has been dealt in the
works {[}18{]}, {[}19{]}-- with respect to the Monge-Kantorovich
distance and in {[}10{]}-- with respect to the Schatten norms. Interestingly,
the convergence obtained in {[}19{]} is even global in time. Other
related developments include {[}8{]} and {[}9{]} which consider
a semi-classical mean field limit for Boson systems, and {[}5{]},
{[}11{]}, {[}20{]} and {[}26{]}, which consider a semi-classical
mean field limit for Fermionic systems. Less is known about how the
inclusion of a magnetic field influences the semi-classical asymptotic
(even when the electric potential is sufficiently nice). The work
{[}1{]} proves weak convergence of the Wigner transform associated
to the (magnetic) von Neumann equation to the solution of Liouville
equation (\ref{Vlasov}), whereas in the present work the convergence
is with respect to the Monge-Kantorovich distance of exponent 2 and
is moreover quantitative. In the context of semi-classical mean-field
limits, we mention the work {[}24{]} which derives the Hartree equation
as a mean-field limit of the $N$-body Schr\textcyr{\"\cyro}dinger
equation with Coulomb potential and a magnetic field by employing
methods of second quantization. However, this convergence is not uniform
in the Planck constant-- uniform convergence with respect to $\hbar$
remains open for electric potentials with Coulomb singularity, both
in the presence and absence of a magnetic field. Another work which
is closely related to combined semi-classical mean field limit of
the quantum many body problem with magnetic field is {[}25{]}. 

\medskip{}
This paper can be viewed as an additional step towards a deeper understanding
of magnetic semi-classical limits. The regularity assumptions that
we impose on the electric potential are identical to the conditions
in {[}14{]}. Extending the result to electric potentials with Coulomb
singularity is left for future investigation. We adapt the methods
introduced in {[}14{]} in order to obtain Monge-Kantorovich convergence
when including a Lipschitz continuous magnetic vector potential $A(x)$
with Lipschitz gradient. In Section (\ref{sec:3 OF CHAP 3}) the same
problem will be investigated, but with a 2D magnetic vector potential
which carries the specific form $A(x)=\frac{1}{\epsilon}x^{\bot}$
(here $x^{\bot}\coloneqq(-x_{2},x_{1})$, and in this case the limit
will be shown to be uniform both in $\epsilon$ and $\hbar$-- which
are both viewed as small parameters. In Section (\ref{OBSERVATION SECTION })
we will apply these results in order to obtain a quantitative observation
type inequality (necessary background to be reviewed in the sequel)
for the von Neumann equation. Both results are of semi-classical type
since the distance considered compares a quantum object (a solution
of the Hartree or von Neumann equation) with a classical object (solution
of the Vlasov or Liouville equation). 

\medskip{}

Let us now give a rough outline of the idea of the proof in {[}14{]},
which in turn will clarify the contribution of this work in comparison
to the existing literature. For simplicity, let us for the moment
omit the quadratic term $\frac{1}{2}\left|y\right|^{2}$. Our starting
point is the definition of a semi-classical pseudo-distance as described
above. To properly motivate this definition, we recall the definition
of the Monge-Kantorovich distance at the classical level. Given $\mu,\nu\in\mathcal{P}(\mathbb{R}^{d})$,
a \textit{coupling of $\mu$ and $\nu$} is a probability measure
$\pi\in\mathcal{P}(\mathbb{R}^{d}\times\mathbb{R}^{d})$ such that
\[
\pi_{1}=\mu,\ \pi_{2}=\nu,
\]

where $\pi_{1},\pi_{2}$ are the first and second marginals of $\pi$
respectively. The set of all couplings of $\mu$ and $\nu$ is denoted
by $\Pi(\mu,\nu)$. We denote by $\mathcal{P}_{p}(\mathbb{R}^{d})$
the set of Borel probability measures with finite moments of order
$p$ ($1\leq p<\infty$), i.e. 

\[
\mathcal{P}_{p}(\mathbb{R}^{d})\coloneqq\left\{ \mu\in\mathcal{P}(\mathbb{R}^{d})|\underset{\mathbb{R}^{d}}{\int}\left|x\right|^{p}d\mu(x)<\infty\right\} .
\]
For each $\mu,\nu\in\mathcal{P}_{p}(\mathbb{R}^{d})$ the \textit{Monge-Kantorovich
distance of exponent $p$ of $\mu$ and} $\nu$ is defined by 

\[
\mathrm{dist}_{\mathrm{MK},p}(\mu,\nu)\coloneqq\underset{\pi\in\Pi(\mu,\nu)}{\inf}\left(\underset{\mathbb{R}^{d}\times\mathbb{R}^{d}}{\int}\left|x-y\right|^{p}\pi(dxdy)\right)^{\frac{1}{p}}.
\]
We wish now to modify the above definition in a way that would allow
to compare probability densities (which in practice will be solution
to the Vlasov or Liouville equations) and density operators (which
in practice will be solution to the Hartree or von Neumann equations).
Thus, a definition of a coupling between these two objects is sought
after. As a general rule, when moving from classical to quantum, integrals
of functions are replaced by traces of the corresponding operator,
and therefore we are lead to 
\begin{defn}
\begin{flushleft}
(Definition 2.1 in {[}14{]}) Let $f(x,\xi)$ be a a probability density
on $\mathbb{R}^{d}\times\mathbb{R}^{d}$. Let $R$ be a density operator
on $\mathfrak{H}$. A\textit{ coupling }of $f$ and $R$ is a measurable
function $Q:\mathbb{R}^{d}\times\mathbb{R}^{d}\rightarrow\mathcal{L}(\mathfrak{H})$
such that $Q(x,\xi)=Q(x,\xi)^{\ast}\geq0$ for a.e. $\mathbb{R}^{d}\times\mathbb{R}^{d}$
and 
\par\end{flushleft}

\end{defn}
\begin{center}
(i). $\mathrm{trace}\left(Q(x,\xi)\right)=f(x,\xi)$ for a.e. $(x,\xi)\in\mathbb{R}^{d}\times\mathbb{R}^{d}$ 
\par\end{center}

and 
\begin{center}
(ii). $\underset{\mathbb{R}^{d}\times\mathbb{R}^{d}}{\int}Q(x,\xi)dxd\xi=R.$ 
\par\end{center}

The set of all couplings of $f$ and $R$ is denoted by $\mathcal{C}(f,R)$,
and is nonempty as witnessed by the operator valued function $(x,\xi)\mapsto f(x,\xi)R$.
Note that the integrand on the left hand side of (ii) is an element
of the separable Banach space $\mathcal{L}^{1}(\mathfrak{H})$ (because
of (i)), and in accordance the integral should be interpreted as a
Bochner integral. We also denote by $\mathcal{D}^{2}(\mathfrak{H})$
the space of density operators with finite second quantum moments,
i.e. all $R\in\mathcal{D}(\mathfrak{H})$ such that 

\[
\mathrm{trace}\left(\sqrt{R}(-\hbar^{2}\Delta+\left|y\right|^{2})\sqrt{R}\right)<\infty.
\]
Equipped with the right notion of a coupling enables to mimic the
definition of the Monge-Kantorovich distance as follows. 
\begin{defn}
\label{def 1.2} For each probability density $f=f(x,\xi)$ on $\mathbb{R}^{d}\times\mathbb{R}^{d}$
and each $R\in\mathcal{D}^{2}(\mathfrak{H})$ set 

\[
E_{\hbar}(f,R)\coloneqq\underset{Q\in\mathcal{C}(f,R)\cap\mathcal{D}^{2}(\mathfrak{H})}{\inf}\left(\underset{\mathbb{R}^{d}\times\mathbb{R}^{d}}{\int}\mathrm{trace}\left(\sqrt{Q(x,\xi)}c_{\hbar}(x,\xi)\sqrt{Q(x,\xi)}\right)dxd\xi\right)^{\frac{1}{2}},
\]

where $c_{\hbar}$ is a function of $(x,\xi)\in\mathbb{R}^{d}\times\mathbb{R}^{d}$
with values in the set of unbounded operators on $\mathfrak{H}$,
called the\textit{ cost function}, and defined by 
\[
c_{\hbar}(x,\xi)\coloneqq\frac{1}{2}\left|x-y\right|^{2}+\frac{1}{2}\left|\xi+i\hbar\nabla_{y}\right|^{2}=\frac{1}{2}\stackrel[k=1]{d}{\sum}(x_{k}-y_{k})^{2}+\frac{1}{2}\stackrel[k=1]{d}{\sum}(\xi_{k}+i\hbar\partial_{y_{k}})^{2}.
\]
The differential operator $c_{\hbar}(x,\xi)$ is a semi-classical
version of the cost function from optimal transport, and is viewed
as an operator indexed by $(x,\xi)\in\mathbb{R}^{d}\times\mathbb{R}^{d}$
acting on the space $L^{2}(\mathbb{R}_{y}^{d})$. The infimum is restricted
to couplings in $\mathcal{D}^{2}(\mathfrak{H})$ in order to avoid
ambiguity in the definition of the trace. Consider the time dependent
quantity 
\[
\mathcal{E}_{\hbar}(t)\coloneqq\underset{\mathbb{R}^{d}\times\mathbb{R}^{d}}{\int}\mathrm{trace}\left(\sqrt{Q_{\hbar}(t,x,\xi)}c_{\hbar}(x,\xi)\sqrt{Q_{\hbar}(t,x,\xi)}\right)dxd\xi
\]
\end{defn}
where $Q_{\hbar}(t,x,\xi)$ is the solution to the Cauchy problem 

\[
\left\{ \begin{array}{c}
\partial_{t}Q_{\hbar}+\left\{ \frac{1}{2}\left|\xi+A(x)\right|^{2}+V(x),Q_{\hbar}\right\} +\frac{i}{\hbar}\left[\frac{1}{2}\left|-i\hbar\nabla_{y}+A(y)\right|^{2}+V(y),Q_{\hbar}\right]=0,\\
Q_{\hbar}(0,x,\xi)=Q_{\hbar}^{\mathrm{in}},
\end{array}\right.
\]
for $Q_{\hbar}^{\mathrm{in}}\in\mathcal{C}(f^{\mathrm{in}},R^{\mathrm{in}})$.
When no magnetic field is included, i.e. $A\equiv0$, the argument
in {[}14{]} rests upon establishing a Gronwall estimate for $\mathcal{E}_{\hbar}$.
We can formally differentiate $\mathcal{E}_{\hbar}$ in time and use
the cyclic property of the trace (and specifically the identity $\mathrm{trace}\left(T\left[S,R\right]\right)=-\mathrm{trace}\left(R\left[T,S\right]\right)$)
to get 

\[
\dot{\mathcal{E}_{\hbar}}(t)=\underset{\mathbb{R}^{d}\times\mathbb{R}^{d}}{\int}\mathrm{trace}\left(Q_{\hbar}(t,x,\xi)\left\{ \frac{1}{2}\left|\xi\right|^{2}+V(x),c_{\hbar}(x,\xi)\right\} \right)dxd\xi
\]

\[
+\underset{\mathbb{R}^{d}\times\mathbb{R}^{d}}{\int}\mathrm{trace}\left(Q_{\hbar}(t,x,\xi)\frac{i}{\hbar}\left[-\frac{\hbar^{2}}{2}\Delta_{y}+V(y),c_{\hbar}(x,\xi)\right]\right)dxd\xi.
\]

The core of the proof in {[}14{]} leading to the desired Gronwall
estimate for $\mathcal{E}_{\hbar}(t)$ is in obtaining the following
operator inequality 

\begin{equation}
\left\{ \frac{1}{2}\left|\xi\right|^{2}+V(x),c_{\hbar}(x,\xi)\right\} +\frac{i}{\hbar}\left[-\frac{\hbar^{2}}{2}\Delta_{y}+V(y),c_{\hbar}(x,\xi)\right]\leq\Lambda c_{\hbar}(x,\xi),\label{eq:-4}
\end{equation}
where $\Lambda$ is some constant which can be chosen to be independent
of $\hbar$. In general, the Poisson brackets of a polynomial (in
the variables $(x,\xi)$) with a second order differential operator
is a second order differential operator and the commutator of second
order differential operators is a third order differential operator.
Since a third order differential operator cannot be controlled by
a second order differential opertaor, this indicates that an ``abstract
nonsense'' argument cannot lead to the inequality (\ref{eq:-4}).
Indeed, the key observation in {[}14{]} is that for the special operators
of interest there is a ``cancellation phenomena'' leading to the
inequality (\ref{eq:-4}). However, when a non-constant magnetic field
is included this cancellation fails, and there is no reason to expect
that the operator 
\[
\left\{ \frac{1}{2}\left|\xi+A(x)\right|^{2}+V(x),c_{\hbar}(x,\xi)\right\} +\frac{i}{\hbar}\left[\frac{1}{2}\left|-i\hbar\nabla_{y}+A(y)\right|^{2}+V(y),c_{\hbar}(x,\xi)\right]
\]
 can be controlled by the cost function $c_{\hbar}(x,\xi)$. The idea
that we wish to convey here is that considering a magnetic cost function
defined by 

\[
\widetilde{c}_{\hbar}(x,\xi)\coloneqq\frac{1}{2}\left|x-y\right|^{2}+\frac{1}{2}\left|\xi+A(x)-A(y)+i\hbar\nabla_{y}\right|^{2}
\]
enables us to overcome this obstacle, since the functional associated
to it 
\[
\widetilde{\mathcal{E}_{\hbar}}(t)\coloneqq\underset{\mathbb{R}^{d}\times\mathbb{R}^{d}}{\int}\mathrm{trace}_{\mathfrak{H}}\left(\sqrt{Q_{\hbar}(t,x,\xi)}\widetilde{c}_{\hbar}(x,\xi)\sqrt{Q_{\hbar}(t,x,\xi)}\right)dxd\xi
\]
can be easily shown to be dominated by $\mathcal{E}_{\hbar}$. We
should remark though that this equivalence between $\mathcal{E}_{\hbar}$
and $\widetilde{\mathcal{E}_{\hbar}}$ is not uniform with respect
to the Lipschitz constant of $A$, and is therefore not suitable for
the case where $A(x)=\frac{1}{\epsilon}x^{\bot}$ and $\epsilon$
is taken to $0$. This regime is addressed in section (\ref{sec:3 OF CHAP 3}). 

\medskip{}
The next section is aimed at fixing the notations, recalling necessary
background and stating the main results. The main results are then
proved in Section \ref{sec:5 OF CHAP 3} and Section \ref{sec:3 OF CHAP 3}.
How to deduce Monge-Kantorovich convergence from the evolution estimates
of the latter sections is explained in Section \ref{sec:Monge-Kantorovich-Convergence SECTION}.
Section \ref{OBSERVATION SECTION } is an application of the main
results to observation inequalities. Finally, Section \ref{domainS SECTION}
further elaborates on some spectral theory subtleties which are created
due to the magnetic field, thereby putting the conceptual argument
on rigorous grounds. 

\section*{Acknowledgment }

This paper is part of the author's work towards a PhD. The author
would like to express his deepest gratitude towards his supervisor
Fran\textcyr{\cyrsdsc}ois Golse for suggesting this problem, for many
fruitful discussions and for carefully reading previous versions of
this manuscript and providing insightful comments and improvements.
The author also thanks Bernard Helffer for a helpful correspondence
regarding the spectral theory of magnetic Schr\textcyr{\"\cyro}dinger
operators, as well as Thierry Paul for clarifications related to the
theory of semi-classical optimal transport. The author is grateful
to Christof Sparber for drawing his attention to the reference {[}1{]}.
The author is indebted to the anonymous referees for providing many
comments which improved dramatically the quality of this work. This
work was partially supported by the research grant ``Stability analysis
for nonlinear partial differential equations across multiscale applications''. 

\section{Preliminaries and Main Results }

We start by fixing some notations from Hilbert space theory. Set $\mathfrak{H}\coloneqq L^{2}(\mathbb{R}^{d})$.
As customary, $\mathcal{L}(\mathfrak{H})$ stands for the normed space
of bounded linear operators on $\mathfrak{H}$. We denote by $\mathcal{D}(\mathfrak{H})$
the set of density operators, i.e. all $R\in\mathcal{L}(\mathfrak{H})$
such that 
\[
R=R^{\ast}\geq0,\ \mathrm{trace}_{\mathfrak{H}}(R)=1,
\]
 and by $\mathcal{D}_{A}^{2}(\mathfrak{H})$, with the abbreviation
$\mathcal{D}^{2}(\mathfrak{H})\coloneqq\mathcal{D}_{0}^{2}(\mathfrak{H})$,
the set of all $R\in\mathcal{D}(\mathfrak{H})$ such that 
\[
\mathrm{trace}\left(\sqrt{R}\left(\left|-i\hbar\nabla_{y}+A(y)\right|^{2}+\left|y\right|^{2}\right)\sqrt{R}\right)<\infty.
\]

The magnetic vector potential $A\in C^{\infty}(\mathbb{R}^{d})$ will
always be taken to be: 

$(\mathbb{A})$. If $d=2$: $A(x)=\frac{1}{\epsilon}x^{\bot}$.

$(\mathbb{A}')$. If $d\geq3$: Lipschitz with Lipschitz gradient,
i.e. $\left|A(x)-A(y)\right|\leq K\left|x-y\right|$ and $\left|\nabla A(x)-\nabla A(y)\right|\leq K'\left|x-y\right|$
for some constant $K,K'>0$. In addition assume without loss of generality
$A(0)=0$. 

As we will see in Section (\ref{domainS SECTION}), $\mathcal{D}_{A}^{2}(\mathfrak{H})=\mathcal{D}^{2}(\mathfrak{H})$
whenever $A$ is sublinear (for $d\geq3$) or $A(x)=\frac{1}{\epsilon}x^{\bot}$
($d=2$). The distinction between 2D and arbitrary dimension $d$
should not be regarded as an essential point because according to
Remark 2.11 in {[}17{]}, it is possible to formulate a general statement
unifying all dimensions $d\geq2$, but we were unable to locate a
reference with a full proof of this. In the sequel, the electric potential
$V$ is assumed to be real-valued function verifying the following
condition 
\begin{equation}
V(x)=V(-x),\ V\in C^{1,1}(\mathbb{R}^{d}).\label{CONDITION ON V}
\end{equation}

The Poisson brackets of functions $f(x,\xi),g(x,\xi)$ on $\mathbb{R}^{d}\times\mathbb{R}^{d}$
(which may be operator valued) are defined by 
\[
\left\{ f,g\right\} \coloneqq\nabla_{\xi}f\cdot\nabla_{x}g-\nabla_{x}f\cdot\nabla_{\xi}g.
\]
 As customary, $\left[\cdot,\cdot\right]$ is the commutator defined
by 
\[
\left[T,S\right]\coloneqq TS-ST,
\]
and $\lor$ is the anti-commutator defined by 
\[
T\lor S\coloneqq TS+ST.
\]
We will see that $E_{\hbar}(f,R)$ (see Definition \ref{def 1.2})
can be bounded from below (and above in case $R$ carries the form
of a special operator called Toeplitz operator) in terms of the Monge-Kantorovich
distance. In Section \ref{sec:Monge-Kantorovich-Convergence SECTION}
we will explain how $E_{\hbar}$ is related to the Husimi transform
and the Monge-Kantorovich distance $\mathrm{dist}_{\mathrm{MK},2}$,
and therefore we now turn our attention to briefly review some elementary
definitions and facts from the theory of Husimi transforms and their
relatives. Let $z=(p,q)\in\mathbb{R}^{d}\times\mathbb{R}^{d}$. For
each $\hbar$ we consider the complex valued $L^{2}$-function on
$\mathbb{R}^{d}$ (called the coherent state) defined by 

\[
\left|z,\hbar\right\rangle (x)\coloneqq(\pi\hbar)^{-\frac{d}{4}}e^{-\frac{(x-q)^{2}}{2\hbar}}e^{\frac{ip\cdot x}{\hbar}}.
\]
We denote by $\left.\left|z,\hbar\right\rangle \right|\left\langle z,\hbar\right|:\mathfrak{H}\rightarrow\mathfrak{H}$
the orthogonal projection on the line $\mathbb{C}\left|z,\hbar\right\rangle $
in $\mathfrak{H}$. For each finite or positive Borel measure $\mu$
on $\mathbb{R}^{d}\times\mathbb{R}^{d}$ we define the Toeplitz operator
$\mathrm{OP}_{\hbar}^{T}(\mu):\mathfrak{H}\rightarrow\mathfrak{H}$
at scale $\hbar$ with symbol $\mu$ by the formula 

\begin{equation}
\mathrm{OP}_{\hbar}^{T}(\mu)\coloneqq\frac{1}{(2\pi\hbar)^{d}}\underset{\mathbb{R}^{d}\times\mathbb{R}^{d}}{\int}\left.\left|z,\hbar\right\rangle \right|\left\langle z,\hbar\right|\mu(dz).\label{eq:definition of Toeplitz}
\end{equation}
The operator $\left.\left|z,\hbar\right\rangle \right|\left\langle z,\hbar\right|$
is trace class (as a rank 1 self-adjoint projection) and the integral
should be interpreted as a Bochner integral, while the operator $\mathrm{OP}_{\hbar}^{T}(\mu)$
is possibly unbounded. We also recall the definition of the Wigner
and Husimi transforms. 
\begin{defn}
For $R$ an unbounded operator on $\mathfrak{H}$ with integral kernel
$r\in\mathcal{S}'(\mathbb{R}^{d}\times\mathbb{R}^{d})$, the \textit{Wigner
transform of scale $\hbar$ }of $R$ is the distribution on $\mathbb{R}^{d}\times\mathbb{R}^{d}$
defined by the formula 
\[
W_{\hbar}[R](x,\xi)\coloneqq\frac{1}{(2\pi)^{d}}\mathcal{F}_{2}(r\circ j_{\hbar})(x,\xi),
\]
where $j_{\hbar}(y,\eta)\coloneqq(y+\frac{1}{2}\hbar\eta,y-\frac{1}{2}\hbar\eta)$
and $\mathcal{F}_{2}$ stands for the Fourier transform with respect
to the second variable. The \textit{Husimi transform of scale $\hbar$
}is the function on $\mathbb{R}^{d}\times\mathbb{R}^{d}$ defined
by the formula 
\[
\widetilde{W_{\hbar}}[R](x,\xi)\coloneqq e^{\frac{\hbar\Delta_{x,\xi}}{4}}W_{\hbar}[R](x,\xi).
\]
\end{defn}
\begin{rem}
Denoting by $G_{a}^{d}$ the centered Gaussian density on $\mathbb{R}^{d}$
with covariance matrix $aI$ we can equivalently write 

\[
\widetilde{W_{\hbar}}[R](x,\xi)=G_{\frac{\hbar}{2}}^{2d}\star(W_{\hbar}[R])(x,\xi).
\]
If $\mu$ is a finite or positive Borel measure on $\mathbb{R}^{d}\times\mathbb{R}^{d}$
then 
\[
W_{\hbar}[\mathrm{OP}_{\hbar}^{T}(\mu)]=\frac{1}{(2\pi\hbar)^{d}}G_{\frac{\hbar}{2}}^{2d}\star\mu
\]
(see formula (51) in {[}13{]}). 
\end{rem}
The following formula relates the Toeplitz operator to the Husimi
transform 
\begin{thm}
\textup{\label{Toeplitz Husimi } ({[}13{]}, Formula (54)).} If $R\in\mathcal{D}(\mathfrak{H})$
and $\mu$ is a probability measure on $\mathbb{R}^{d}\times\mathbb{R}^{d}$
then $\mathrm{OP}_{\hbar}^{T}(\mu)$ is trace class and 

\begin{equation}
\mathrm{trace}\left(\mathrm{OP}_{\hbar}^{T}(\mu)R\right)=\underset{\mathbb{R}^{d}\times\mathbb{R}^{d}}{\int}\widetilde{W_{\hbar}}[R](x,\xi)\mu(dxd\xi).\label{eq:-11-1}
\end{equation}
\end{thm}
An additional functional analytic tool which we freely use is Kato's
theory of the self-adjointness of perturbed self adjoint unbounded
operators.\footnote{The power of Kato's theory is reflected mainly in its capacity of
handeling singular perturbations such as the Coulomb potential. Since
the electric potential is assumed to be fairly regular, Kato's theory
is not strictly necessary here. } We denote by $\mathscr{H}$ the quantum Hamiltonian defined by 

\[
\mathscr{H}\coloneqq\frac{1}{2}\left|-i\hbar\nabla_{y}+A(y)\right|^{2}+\frac{1}{2}\left|y\right|^{2}+V(y)\coloneqq\frac{1}{2}\mathscr{H}_{0}+V(y).
\]
The operator $\mathscr{H}_{0}$ is called the \textit{quantum kinetic
energy} and plays the role of the perturbed operator while the potential
$V$ plays the role of the perturbation. The reason why $\mathscr{H}_{0}$
is an essentially self-adjoint operator on $C_{0}^{\infty}(\mathbb{R}^{d})$
is explained Section \ref{domainS SECTION}. By Kato's theorem (see
e.g. Theorem 6.4 in {[}28{]}), in order to assert that $\mathscr{H}_{0}$
is essentially self-adjoint on $C_{0}^{\infty}(\mathbb{R}^{d})$ it
is sufficient to show that $V$ (viewed as a multiplication operator)
is relatively bounded, with relative bound $<1$, with respect to
$\mathscr{H}_{0}$-- which is trivially true since $V$ is bounded.
By Stone's theorem, that $\mathscr{H}$ is essentially self-adjoint
implies that the operator $U(t)\coloneqq e^{\frac{it\mathscr{H}}{\hbar}}$
is unitary. The same considerations permit to view $c_{\hbar}^{\lambda}(x,\xi),\widetilde{c}_{\hbar}^{\lambda}(x,\xi)$
as essentially self-adjoint operators with domains $D(c_{\hbar}^{\lambda}(x,\xi)),D(\widetilde{c_{\hbar}^{\lambda}}(x,\xi))$
respectively such that 
\[
D\left(-\frac{\hbar^{2}}{2}\Delta+\frac{1}{2}\left|y\right|^{2}\right)\subset D(c_{\hbar}^{\lambda}(x,\xi))
\]
 and 
\[
D\left(\frac{1}{2}\left|-i\hbar\nabla_{y}+A(y)\right|^{2}+\frac{1}{2}\left|y\right|^{2}\right)\subset D(\widetilde{c}_{\hbar}^{\lambda}(x,\xi)).
\]
See Section \ref{domainS SECTION} for more useful information about
the spectral theory of the magnetic/non-magnetic harmonic oscillator.
At this stage, it is worth mentioning that the problem of finding
a unique self-adjoint extension for magnetic Hamiltonians has been
investigated in {[}3{]}, where the authors prove a Kato type theorem
under significantly weaker assumptions on $V$ and $A$. We can now
state our main theorems. 
\begin{thm}
\label{Main theorem 2} Let $V$ satisfy (\ref{CONDITION ON V}) and
$V+\frac{1}{2}\left|x\right|^{2}\geq0$. Set $L\coloneqq\mathrm{Lip}(\nabla V)$.
Let $A:\mathbb{R}^{d}\rightarrow\mathbb{R}^{d}$ be a vector field
satisfying $(\mathbb{A}')$ with constants $d,K$ and $K'$. Let $f^{\mathrm{in}}$
be a probability density on $\mathbb{R}^{d}\times\mathbb{R}^{d}$
with finite second moments, i.e 

\[
\underset{\mathbb{R}^{d}\times\mathbb{R}^{d}}{\int}\left(\left|x\right|^{2}+\left|\xi\right|^{2}\right)f^{\mathrm{in}}(x,\xi)dxd\xi<\infty.
\]

Assume in addition there is some $\rho_{0}>0$ such that 

\[
\mathrm{supp}(f^{\mathrm{in}})\subset\left\{ (x,\xi)\left|\frac{1}{2}\left|\xi+A(x)\right|^{2}+\frac{1}{2}\left|x\right|^{2}+V(x)\leq\rho_{0}^{2}\right.\right\} .
\]
Let $f(t,\cdot)$ be the solution of the Liouville equation (\ref{Vlasov})
with initial data $f^{\mathrm{in}}$. Let $R_{\hbar}^{\mathrm{in}}=\mathrm{OP}_{\hbar}^{T}((2\pi\hbar)^{d}\mu^{\mathrm{in}})$
where $\mu^{\mathrm{in}}$ is a Borel probability measure on $\mathbb{R}^{d}\times\mathbb{R}^{d}$
with finite second moments, and let $R_{\hbar}(t)$ be the solution
to the von Neumann equation (\ref{Hartree}) with initial data $R_{\hbar}^{\mathrm{in}}$.
Then for all $t\in[0,\tau]$ 
\[
\mathrm{dist}_{\mathrm{MK},2}\left(f(t,\cdot),\widetilde{W}_{\hbar}[R_{\hbar}(t)]\right)^{2}\leq\beta(K)e^{\alpha(L,K',\rho_{0})t}\left(\mathrm{dist}_{\mathrm{MK},2}(f^{\mathrm{in}},\mu^{\mathrm{in}})^{2}+\frac{d\hbar}{2}\right)+\frac{d\hbar}{2}
\]
where 

\[
\alpha(L,K',\rho_{0})\coloneqq2+L^{2}+4\rho_{0}\max(2K'^{2},1)
\]
and 

\[
\beta(K)\coloneqq\max(2,1+2K^{2})^{2}.
\]
\end{thm}
\begin{rem}
As we will see, the initial finite second moments assumption on $f^{\mathrm{in}}$
and the assumption $R^{\mathrm{in}}\in\mathcal{D}^{2}(\mathfrak{H})$
are propagated in time. Consequently (see appendix B in {[}13{]}
and especially formulas (54) and (48)), $\widetilde{W}_{\hbar}[R_{\hbar}(t)]$
has finite second moments and therefore can indeed be compared with
$f(t,\cdot)$ through $\mathrm{dist}_{\mathrm{MK},2}$. This remark
is also relevant for Theorem \ref{Main thm } below. 
\end{rem}
Our second result concerns a similar question, but with magnetic vector
potential which corresponds to planar rotation by $\frac{\pi}{2}$.
Namely, we take $A(x)\coloneqq\frac{1}{\epsilon}x^{\bot}$ and we
are interested in the limit both as $\epsilon\rightarrow0,\hbar\rightarrow0$. 
\begin{thm}
\label{Main thm } Let $A(x)=\frac{1}{\epsilon}x^{\bot}$ and let
$V$ satisfy (\ref{CONDITION ON V}). Set $L\coloneqq\mathrm{Lip}(\nabla V)$.
Let $f_{\epsilon}^{\mathrm{in}}$ be a probability density on $\mathbb{R}^{2}\times\mathbb{R}^{2}$
with finite second moments. Let $f_{\epsilon}(t,\cdot)$ be the solution
of the Liouville equation (\ref{Vlasov}) with initial data $f_{\epsilon}^{\mathrm{in}}$.
Let $R_{\hbar,\epsilon}^{\mathrm{in}}=\mathrm{OP}_{\hbar}^{T}((2\pi\hbar)^{2}\mu_{\epsilon}^{\mathrm{in}})$
where $\mu_{\epsilon}^{\mathrm{in}}$ is a Borel probability measure
on $\mathbb{R}^{2}\times\mathbb{R}^{2}$ with finite second moments,
and let $R_{\hbar,\epsilon}(t)$ be the solution to the von Neumann
equation (\ref{Hartree}) with initial data $R_{\hbar,\epsilon}^{\mathrm{in}}$.
Then for all $t\in[0,\tau]$ 
\[
\mathrm{dist}_{\mathrm{MK},2}\left(f_{\epsilon}(t,\cdot),\widetilde{W_{\hbar}}[R_{\hbar,\epsilon}]\right)^{2}\leq e^{Ct}\left(\mathrm{dist}_{\mathrm{MK},2}(f_{\epsilon}^{\mathrm{in}},\mu_{\epsilon}^{\mathrm{in}})^{2}+\frac{1}{2}(1+\epsilon^{2})\hbar\right)+\frac{1}{2}(1+\epsilon^{2})\hbar
\]

with $C\coloneqq\max(2,\epsilon^{2}(1+L^{2}))$. 
\end{thm}
\begin{rem}
Explicitly, the solution of equation (\ref{Hartree}) can be written
as $R(t)=U^{\ast}(t)R^{\mathrm{in}}U(t)$, whereas the solution of
equation (\ref{Vlasov}) is obtained as the push-forward under the
flow of Newton's second order system of ODEs (whose existence and
uniqueness is guaranteed by the Cauchy-Lipschitz theorem). Namely,
if $\Phi_{t}(x,\xi)=(X(t),\Xi(t))$ is the flow of the system 

\begin{equation}
\left\{ \begin{array}{cc}
\frac{d}{dt}X(t)=\Xi(t)+A(X(t)) & X(0)=x\\
\frac{d}{dt}\Xi(t)=-\nabla A(X(t))(\Xi+A(X))-X(t)-\nabla V(X(t)) & \Xi(0)=\xi
\end{array}\right.,\label{Newton magnetic}
\end{equation}
then $f(t,x,\xi)=f^{\mathrm{in}}(\Phi_{-t}(x,\xi))$. In the special
case where $A(x)=\frac{1}{\epsilon}x^{\bot}$, the system (\ref{Newton magnetic})
becomes 
\begin{equation}
\left\{ \begin{array}{cc}
\frac{d}{dt}X(t)=\Xi(t)+\frac{1}{\epsilon}X(t) & X(0)=x\\
\frac{d}{dt}\Xi(t)=\frac{1}{\epsilon}\Xi^{\bot}(t)-\frac{1}{\epsilon^{2}}X(t)-X(t)-\nabla V(X(t)) & \Xi(0)=\xi
\end{array}\right..\label{eq:Newton constant magnetic}
\end{equation}
\end{rem}
\begin{rem}
Both statements of Theorem \ref{Main thm } and \ref{Main theorem 2}
are valid (with some appropriate modifications) for the Hartree and
Vlasov equations. In this case, the estimates for the terms arising
from the electric potential are slightly more involved due to the
convolutional non-linearity. See Theorem 2.5 in {[}14{]} for a guidance
on how this is to be done. 
\end{rem}

\section{\label{sec:5 OF CHAP 3} Arbitrary Non-constant Magnetic Field }

In this section we consider the question of deriving the Liouville
equation from the von Neumann equation in the presence of a magnetic
vector potential verifying the assumption $(\mathbb{A}')$. Recall
that $f,R_{\hbar}$ are the solutions to equations (\ref{Vlasov})
and (\ref{Hartree}) with magnetic vector potential $A(x)$ and with
initial data $f^{\mathrm{in}},R^{\mathrm{in}}$, respectively. Let
$Q^{\mathrm{in}}\in\mathcal{C}(f^{\mathrm{in}},R^{\mathrm{in}})\cap\mathcal{D}^{2}(\mathfrak{H})$
and let $Q=Q(t,x,\xi)$ be defined by

\[
Q(t,x,\xi)\coloneqq U(t)^{\ast}Q^{\mathrm{in}}(\Phi_{-t}(x,\xi))U(t).
\]
Alternatively $Q(t,x,\xi)$ is the unique solution of the Cauchy problem
for the semi-classical coupling equation 

\[
\left\{ \begin{array}{c}
\partial_{t}Q_{\hbar}+\left\{ \frac{1}{2}\left|\xi+A(x)\right|^{2}+V(x),Q_{\hbar}\right\} +\frac{i}{\hbar}\left[\frac{1}{2}\left|-i\hbar\nabla_{y}+A(y)\right|^{2}+V(y),Q_{\hbar}\right]=0,\\
Q_{\hbar}(0,x,\xi)=Q_{\hbar}^{\mathrm{in}}\in\mathcal{C}(f^{\mathrm{in}},R^{\mathrm{in}})\cap\mathcal{D}^{2}(\mathfrak{H}).
\end{array}\right.
\]
Eventually we wish to obtain an evolution estimate on the time dependent
quantity 

\begin{equation}
\mathcal{E}_{\hbar}^{\lambda}(t)\coloneqq\underset{\mathbb{R}^{d}\times\mathbb{R}^{d}}{\int}\mathrm{trace}_{\mathfrak{H}}\left(\sqrt{Q(t,x,\xi)}c_{\hbar}^{\lambda}(x,\xi)\sqrt{Q(t,x,\xi)}\right)dxd\xi\geq0,\label{eq:definition 1}
\end{equation}
where 

\[
c_{\hbar}^{\lambda}(x,\xi)\coloneqq\frac{1}{2}\lambda^{2}\left|x-y\right|^{2}+\frac{1}{2}\left|\xi+i\hbar\nabla_{y}\right|^{2}.
\]
We insert a parameter $\lambda>0$ in the cost function in order to
optimize some constants which will show up in Section (\ref{OBSERVATION SECTION }).
However, for the purpose of establishing Theorem \ref{Main theorem 2}
this is unnecessary, and so the reader is advised to follow the forthcoming
calculation with $\lambda=1$. We will define an auxiliary functional
$\widetilde{\mathcal{E}{}_{\hbar}^{\lambda}}(t)$ by 

\begin{equation}
\widetilde{\mathcal{E}_{\hbar}^{\lambda}}(t)\coloneqq\underset{\mathbb{R}^{d}\times\mathbb{R}^{d}}{\int}\mathrm{trace}_{\mathfrak{H}}\left(\sqrt{Q(t,x,\xi)}\widetilde{c}_{\hbar}^{\lambda}(x,\xi)\sqrt{Q(t,x,\xi)}\right)dxd\xi\geq0,\label{eq:definition 2}
\end{equation}
where 
\[
\widetilde{c}_{\hbar}^{\lambda}(x,\xi)\coloneqq\frac{\lambda^{2}}{2}\left|x-y\right|^{2}+\frac{1}{2}\left|\xi+A(x)-A(y)+i\hbar\nabla_{y}\right|^{2}.
\]
Recall that $Q$ propagates in time the coupling property 
\begin{lem}
\begin{flushleft}
\textup{\label{coupling propagates} (Lemma 4.2 in {[}14{]})} With
the same notations and assumptions of Theorem (\ref{Main theorem 2}),
if $Q^{\mathrm{in}}\in\mathcal{C}(f^{\mathrm{in}},R^{\mathrm{in}})$
then $Q(t)\in\mathcal{C}(f(t),R(t))$ for all $t\geq0$. 
\par\end{flushleft}
\end{lem}

\subsection{Classical/Quantum Finite Second Moments are Propagated in Time}

We proceed by proving that equations (\ref{Vlasov}) and (\ref{Hartree})
propagate in time finite second moments, i.e that for all $t\in[0,\tau]$,
$f(t)$ has finite second moments and $R(t)\in\mathcal{D}^{2}(\mathfrak{H})$,
provided this is true for $t=0$. As we will see, both of these observations
justify the finiteness of $\mathcal{E}_{\hbar}^{\lambda}(t)$ and
$\widetilde{\mathcal{E}_{\hbar}^{\lambda}}(t)$ for all times. 
\begin{lem}
\begin{flushleft}
\label{finite second moments } Let $f^{\mathrm{in}}$ be a probability
density on $\mathbb{R}^{d}\times\mathbb{R}^{d}$ such that $\underset{\mathbb{R}^{d}\times\mathbb{R}^{d}}{\int}\left(\left|x\right|^{2}+\left|\xi\right|^{2}\right)f^{\mathrm{in}}(x,\xi)dxd\xi<\infty$.
Let $f$ be a solution of the Cauchy problem (\ref{Vlasov}) with
initial data $f^{\mathrm{in}}$. There is a constant $C>0$ such that
for all $t\in[0,\tau]$
\par\end{flushleft}
\begin{flushleft}
\[
\underset{\mathbb{R}^{d}\times\mathbb{R}^{d}}{\int}\left(\left|x\right|^{2}+\left|\xi\right|^{2}\right)f(t,x,\xi)dxd\xi\leq C\left(\underset{\mathbb{R}^{d}\times\mathbb{R}^{d}}{\int}\left(\left|x\right|^{2}+\left|\xi\right|^{2}\right)f^{\mathrm{in}}(x,\xi)dxd\xi+\left\Vert V\right\Vert _{\infty}\right).
\]
\par\end{flushleft}
\end{lem}
\begin{proof}
The proof is a straightforward modification of Lemma 4.1 in {[}14{]}.
Let $X(t)$ and $\Xi(t)$ be the solutions of equation (\ref{Newton magnetic}).
By conservation of energy for the Liouville equation we have 

\[
\frac{d}{dt}\underset{\mathbb{R}^{d}\times\mathbb{R}^{d}}{\int}\left(\frac{1}{2}\left|\Xi(t)+A(X(t))\right|^{2}+\frac{1}{2}\left|X(t)\right|^{2}+V(X(t))\right)f^{\mathrm{in}}(x,\xi)dxd\xi=0,
\]
which entails 
\[
\frac{1}{2}\underset{\mathbb{R}^{d}\times\mathbb{R}^{d}}{\int}\left(\left|\Xi(t)+A(X(t))\right|^{2}+\left|X(t)\right|^{2}\right)f^{\mathrm{in}}(x,\xi)dxd\xi
\]

\[
=\frac{1}{2}\underset{\mathbb{R}^{d}\times\mathbb{R}^{d}}{\int}\left(\left|\xi+A(x)\right|^{2}+\left|x\right|^{2}\right)f^{\mathrm{in}}(x,\xi)dxd\xi+\underset{\mathbb{R}^{d}\times\mathbb{R}^{d}}{\int}\left(V(x)-V(X(t))\right)f^{\mathrm{in}}(x,\xi)dxd\xi.
\]
Clearly, for all $t\geq0$

\[
\left|\underset{\mathbb{R}^{d}\times\mathbb{R}^{d}}{\int}\left(V(x)-V(X(t))\right)f^{\mathrm{in}}(x,\xi)dxd\xi\right|\leq2\left\Vert V\right\Vert _{\infty},
\]
and therefore 
\[
\frac{1}{2}\underset{\mathbb{R}^{d}\times\mathbb{R}^{d}}{\int}\left(\left|\Xi(t)+A(X(t))\right|^{2}+\left|X(t)\right|^{2}\right)f^{\mathrm{in}}(x,\xi)dxd\xi
\]

\[
\leq\frac{1}{2}\underset{\mathbb{R}^{d}\times\mathbb{R}^{d}}{\int}\left(\left|\xi+A(x)\right|^{2}+\left|x\right|^{2}\right)f^{\mathrm{in}}(x,\xi)dxd\xi+2\left\Vert V\right\Vert _{\infty}
\]

\begin{equation}
\leq\underset{\mathbb{R}^{d}\times\mathbb{R}^{d}}{\int}\left(\left|\xi\right|^{2}+\left(\frac{2K^{2}+1}{2}\right)\left|x\right|^{2}\right)f^{\mathrm{in}}(x,\xi)dxd\xi+2\left\Vert V\right\Vert _{\infty}.\label{eq:-15-2}
\end{equation}
Hence 

\[
\underset{\mathbb{R}^{d}\times\mathbb{R}^{d}}{\int}\left(\left|\xi+A(x)\right|^{2}+\left|x\right|^{2}\right)f(t,x,\xi)dxd\xi
\]

\[
=\underset{\mathbb{R}^{d}\times\mathbb{R}^{d}}{\int}\left(\left|X(t)\right|^{2}+\left|\Xi(t)+A(X(t))\right|^{2}\right)f^{\mathrm{in}}(x,\xi)dxd\xi
\]

\[
\leq C\left(\underset{\mathbb{R}^{d}\times\mathbb{R}^{d}}{\int}\left(\left|x\right|^{2}+\left|\xi\right|^{2}\right)f^{\mathrm{in}}(x,\xi)dxd\xi+\left\Vert V\right\Vert _{\infty}\right),
\]
which concludes the proof.
\end{proof}
We recall the following observation, which will be freely used in
the sequel. 
\begin{lem}
\textup{(Lemma 2.3 in {[}16{]})} If $S=S^{\ast}\geq0$ is unbounded
with domain $D(S)\subset\mathfrak{H}$ and $T=T^{\ast}\geq0$ is trace
class with eigenvectors $(e_{k})_{k\geq1}\subset D(S)$ and eigenvalues
$(\alpha_{k})_{k\geq1}$ respectively, then $\mathcal{L}(\mathfrak{H})\ni\sqrt{T}S\sqrt{T}=\left(\sqrt{T}S\sqrt{T}\right)^{\ast}\geq0$
and 
\[
\mathrm{trace}\left(\sqrt{T}S\sqrt{T}\right)=\stackrel[k=1]{\infty}{\sum}\alpha_{k}\left\langle e_{k},Se_{k}\right\rangle .
\]
\end{lem}
\begin{lem}
\label{finite second quantum moments } Let $R(t)$ be the solution
to the Cauchy problem (\ref{Hartree}) with initial data $R^{\mathrm{in}}\in\mathcal{D}^{2}(\mathfrak{H})$.
Then $R(t)\in\mathcal{D}^{2}(\mathfrak{H})$ for all $t\in[0,\tau]$. 
\end{lem}
\textit{Proof}. By the Hilbert-Schmidt theorem, let $(e_{k})_{k\geq1}$
be a complete system of eigenvectors of $R^{in}$. Because $\mathscr{H}-V=\mathscr{H}_{0}$
we have 

\[
\left\langle e_{k},\sqrt{R(t)}\mathscr{H}_{0}\sqrt{R(t)}e_{k}\right\rangle =\left\langle e_{k},U(t)^{\ast}\sqrt{R^{\mathrm{in}}}U(t)\mathscr{H}_{0}U(t)^{\ast}\sqrt{R^{\mathrm{in}}}U(t)e_{k}\right\rangle 
\]

\[
=\left\langle e_{k},U(t)^{\ast}\sqrt{R^{\mathrm{in}}}U(t)(\mathscr{H}-V)U(t)^{\ast}\sqrt{R^{\mathrm{in}}}U(t)e_{k}\right\rangle 
\]

\[
=\left\langle e_{k},U(t)^{\ast}\sqrt{R^{\mathrm{in}}}\mathscr{H}\sqrt{R^{\mathrm{in}}}U(t)e_{k}\right\rangle -\left\langle e_{k},U(t)^{\ast}\sqrt{R^{\mathrm{in}}}V\sqrt{R^{\mathrm{in}}}U(t)e_{k}\right\rangle 
\]

\[
=\left\langle U(t)e_{k},\sqrt{R^{\mathrm{in}}}\mathscr{H}\sqrt{R^{\mathrm{in}}}U(t)e_{k}\right\rangle -\left\langle U(t)e_{k},\sqrt{R^{\mathrm{in}}}V\sqrt{R^{\mathrm{in}}}U(t)e_{k}\right\rangle .
\]
Since $U^{\ast}(t)e_{k}\in D(\mathscr{H}_{0})=D\left(-\frac{1}{2}\hbar^{2}\Delta+\frac{1}{2}\left|y\right|^{2}\right)$
(see Corollary \ref{domain of magnetic vs non magnetic }), the trace
of $\sqrt{R(t)}\mathscr{H}_{0}\sqrt{R(t)}$ is well defined and 
\[
\mathrm{trace}\left(\sqrt{R(t)}\mathscr{H}_{0}\sqrt{R(t)}\right)
\]

\[
=\stackrel[k=1]{\infty}{\sum}\left\langle e_{k},\sqrt{R(t)}\mathscr{H}_{0}\sqrt{R(t)}e_{k}\right\rangle \leq\stackrel[k=1]{\infty}{\sum}\left\langle U(t)e_{k},\sqrt{R^{\mathrm{in}}}\mathscr{H}\sqrt{R^{\mathrm{in}}}U(t)e_{k}\right\rangle +\left\Vert V\right\Vert _{\infty}
\]

\[
=\mathrm{trace}\left(\sqrt{R^{\mathrm{in}}}\mathscr{H}\sqrt{R^{\mathrm{in}}}\right)+\left\Vert V\right\Vert _{\infty}<\infty,
\]
as desired. 
\begin{onehalfspace}
\begin{flushright}
$\square$
\par\end{flushright}
\end{onehalfspace}
\begin{rem}
\label{remark about Toeplitz operator } The assumption that $\mu^{\mathrm{in}}$
is a Borel probability measure on $\mathbb{R}^{d}\times\mathbb{R}^{d}$
with finite second moments implies that the Toeplitz operator $\mathrm{OP}_{\hbar}^{T}((2\pi\hbar)^{d}\mu^{\mathrm{in}})$
belongs to the space $\mathcal{D}^{2}(\mathfrak{H})$-- see Proposition
2.3 in {[}15{]}. This fact will also be restated explicitly in Section
\ref{sec:Monge-Kantorovich-Convergence SECTION}. 
\end{rem}
We now provide a proof of a magnetic version of Lemma 2.2 in {[}15{]}
which acts as a justification to the fact that the functionals $\mathcal{E}_{\hbar}^{\lambda}(t)$
and $\widetilde{\mathcal{E}{}_{\hbar}^{\lambda}}(t)$ are well defined
for all $t\in[0,\tau]$. Again, the argument here is almost identical
to the one proposed in {[}15{]} . As a preliminary we recall the
following observation 
\begin{prop}
\textup{\label{trace identity } (Lemma 2.1 in {[}14{]}) }Let $T\in\mathcal{L}(\mathfrak{H})$
satisfy $T=T^{\ast}\geq0$ and let $S$ be an unbounded operator such
that $S=S^{\ast}\geq0.$ Then 

\[
\mathrm{trace}\left(\sqrt{T}S\sqrt{T}\right)=\mathrm{trace}\left(\sqrt{S}T\sqrt{S}\right)\in[0,\infty].
\]
\end{prop}
\begin{lem}
\label{Trace estimate} Let $R^{\mathrm{in}}\in\mathcal{D}^{2}(\mathfrak{H})$
and let $f^{\mathrm{in}}$ be a probability density on $\mathbb{R}^{d}\times\mathbb{R}^{d}$
with finite second moments. Suppose $Q^{\mathrm{in}}(x,\xi)\in\mathcal{C}(f^{\mathrm{in}},R^{\mathrm{in}})\cap\mathcal{D}^{2}(\mathfrak{H})$.
Then 
\end{lem}
\[
\underset{\mathbb{R}^{d}\times\mathbb{R}^{d}}{\int}\mathrm{trace}\left(\sqrt{Q(t,x,\xi)}\widetilde{c}_{\hbar}^{\lambda}(x,\xi)\sqrt{Q(t,x,\xi)}\right)dxd\xi
\]

\[
\leq\underset{\mathbb{R}^{d}\times\mathbb{R}^{d}}{\int}\left(\left|x\right|^{2}+\left|\xi\right|^{2}\right)f^{\mathrm{in}}(x,\xi)dxd\xi+\mathrm{trace}\left(\sqrt{R(t)}\left(\left|-i\hbar\nabla_{y}+A(y)\right|^{2}+\lambda^{2}\left|y\right|^{2}\right)\sqrt{R(t)}\right)<\infty.
\]

\begin{proof}
Note the operator inequality 

\[
\widetilde{c_{\hbar}^{\lambda}}(x,\xi)\leq\lambda^{2}\left|x\right|^{2}+\left|\xi+A(x)\right|^{2}+\lambda^{2}\left|y\right|^{2}+\left|-i\hbar\nabla_{y}+A(y)\right|^{2}.
\]

Since $\Phi_{-t}:\mathbb{R}^{d}\times\mathbb{R}^{d}\rightarrow\mathbb{R}^{d}\times\mathbb{R}^{d}$
is a diffeomorphism, we see that $Q^{\mathrm{in}}(\Phi_{-t}(x,\xi))\in\mathcal{D}^{2}(\mathfrak{H})$,
which explains why the trace of $\sqrt{Q(t,x,\xi)}\widetilde{c}_{\hbar}^{\lambda}(x,\xi)\sqrt{Q(t,x,\xi)}$
is well defined. Therefore 

\[
\underset{\mathbb{R}^{d}\times\mathbb{R}^{d}}{\int}\mathrm{trace}\left(\sqrt{Q(t,x,\xi)}\widetilde{c}_{\hbar}^{\lambda}(x,\xi)\sqrt{Q(t,x,\xi)}\right)dxd\xi
\]

\[
\leq\underset{\mathbb{R}^{d}\times\mathbb{R}^{d}}{\int}\mathrm{trace}\left(\sqrt{Q(t,x,\xi)}\left(\lambda^{2}\left|x\right|^{2}+\left|\xi+A(x)\right|^{2}\right)\sqrt{Q(t,x,\xi)}\right)dxd\xi
\]

\[
+\underset{\mathbb{R}^{d}\times\mathbb{R}^{d}}{\int}\mathrm{trace}\left(\sqrt{Q(t,x,\xi)}\left(\lambda^{2}\left|y\right|^{2}+\left|-i\hbar\nabla_{y}+A(y)\right|^{2}\right)\sqrt{Q(t,x,\xi)}\right)dxd\xi.
\]

The first integral is 

\[
\underset{\mathbb{R}^{d}\times\mathbb{R}^{d}}{\int}\left(\lambda^{2}\left|x\right|^{2}+\left|\xi+A(x)\right|^{2}\right)\mathrm{trace}\left(Q(t,x,\xi)\right)dxd\xi\leq\underset{\mathbb{R}^{d}\times\mathbb{R}^{d}}{\int}\left(\lambda^{2}\left|x\right|^{2}+\left|\xi+A(x)\right|^{2}\right)f(t,x,\xi)dxd\xi.
\]

Owing to Proposition \ref{trace identity } the second integral is
recast as 

\[
\underset{\mathbb{R}^{d}\times\mathbb{R}^{d}}{\int}\mathrm{trace}\left(\sqrt{Q(t,x,\xi)}\left(\lambda^{2}\left|y\right|^{2}+\left|-i\hbar\nabla_{y}+A(y)\right|^{2}\right)\sqrt{Q(t,x,\xi)}\right)dxd\xi
\]

\[
=\underset{\mathbb{R}^{d}\times\mathbb{R}^{d}}{\int}\mathrm{trace}\left(\sqrt{\lambda^{2}\left|y\right|^{2}+\left|-i\hbar\nabla_{y}+A(y)\right|^{2}}Q(t,x,\xi)\sqrt{\lambda^{2}\left|y\right|^{2}+\left|-i\hbar\nabla_{y}+A(y)\right|^{2}}\right)dxd\xi
\]

\[
=\mathrm{trace}\left(\sqrt{\lambda^{2}\left|y\right|^{2}+\left|-i\hbar\nabla_{y}+A(y)\right|^{2}}R(t)\sqrt{\lambda^{2}\left|y\right|^{2}+\left|-i\hbar\nabla_{y}+A(y)\right|^{2}}\right)
\]

\[
=\mathrm{trace}\left(\sqrt{R(t)}\left(\lambda^{2}\left|y\right|^{2}+\left|-i\hbar\nabla_{y}+A(y)\right|^{2}\right)\sqrt{R(t)}\right).
\]
\end{proof}

\subsection{The Gronwall Estimate}

We shall first obtain a Gronwall estimate on $\widetilde{\mathcal{E}{}_{\hbar}^{\lambda}}(t)$.
Once this is achieved, the desired estimate for $\mathcal{E}_{\hbar}^{\lambda}(t)$
would follow easily, as implied by the following simple 
\begin{lem}
\label{Equivalance } Let the assumptions of Theorem (\ref{Main theorem 2})
hold. Let $\lambda>0$ and let $\mathcal{E}_{\hbar}^{\lambda}(t)$
and $\widetilde{\mathcal{E}_{\hbar}^{\lambda}}(t)$ be as defined
in (\ref{eq:definition 1}) and (\ref{eq:definition 2}). It holds
that 
\[
\frac{1}{\max\left(2,\frac{\lambda^{2}+2K^{2}}{\lambda^{2}}\right)}\widetilde{\mathcal{E}_{\hbar}^{\lambda}}(t)\leq\mathcal{E}_{\hbar}^{\lambda}(t)\leq\max\left(2,\frac{\lambda^{2}+2K^{2}}{\lambda^{2}}\right)\widetilde{\mathcal{E}_{\hbar}^{\lambda}}(t).
\]
\end{lem}
\begin{proof}
By the triangle inequality 

\[
\widetilde{c}_{\hbar}^{\lambda}(x,\xi)=\frac{\lambda^{2}}{2}\left|x-y\right|^{2}+\frac{1}{2}\left|\xi+A(x)-A(y)+i\hbar\nabla_{y}\right|^{2}
\]

\[
\leq\frac{\lambda^{2}}{2}\left|x-y\right|^{2}+\left|\xi+i\hbar\nabla_{y}\right|^{2}+\left|A(y)-A(x)\right|^{2}
\]

\[
\leq\frac{\lambda^{2}+2K^{2}}{2}\left|x-y\right|^{2}+\left|\xi+i\hbar\nabla_{y}\right|^{2}\leq\max\left(2,\frac{\lambda^{2}+2K^{2}}{\lambda^{2}}\right)c_{\hbar}^{\lambda}(x,\xi),
\]
while

\[
c_{\hbar}^{\lambda}(x,\xi)=\frac{\lambda^{2}}{2}\left|x-y\right|^{2}+\frac{1}{2}\left|\xi+A(x)-A(y)+i\hbar\nabla_{y}-(A(x)-A(y))\right|^{2}
\]

\[
\leq\frac{\lambda^{2}}{2}\left|x-y\right|^{2}+K^{2}\left|x-y\right|^{2}+\left|\xi+A(x)-A(y)+i\hbar\nabla_{y}\right|^{2}
\]

\[
=\frac{\lambda^{2}+2K^{2}}{2}\left|x-y\right|^{2}+\left|\xi+A(x)-A(y)+i\hbar\nabla_{y}\right|^{2}
\]

\[
\leq\max\left(2,\frac{\lambda^{2}+2K^{2}}{\lambda^{2}}\right)\widetilde{c}_{\hbar}^{\lambda}(t,x,\xi).
\]
\end{proof}
We are now ready to prove the following intermediate inequality, which
is the core inequality leading to the estimate announced in Theorem
\ref{Main theorem 2}
\begin{thm}
\begin{flushleft}
\label{Pseudo distance estimates 2} With the same notations and assumptions
of Theorem (\ref{Main theorem 2}) 
\[
E_{\hbar}^{\lambda}(f(t),R_{\hbar}(t))^{2}\leq\beta(K,\lambda)e^{\alpha(L,K',\lambda,\rho_{0})t}E_{\hbar}(f^{in},R_{\hbar}^{in})^{2}
\]
where 
\[
\beta(K,\lambda)\coloneqq\max\left(2,\frac{\lambda^{2}+2K^{2}}{\lambda^{2}}\right)^{2}
\]
and 
\[
\alpha(L,K',\lambda,\rho_{0})\coloneqq1+\max\left(1,\frac{1+L^{2}}{\lambda^{2}}\right)+4\rho_{0}\max\left(\frac{2K'^{2}}{\lambda^{2}},1\right).
\]
\par\end{flushleft}

\end{thm}
The proof below should be regarded as a formal proof. Subtleties such
as differentiability in time of $\widetilde{\mathcal{E}{}_{\hbar}^{\lambda}}(t)$
can be addressed using the eigenfunction expansion method demonstrated
in {[}15{]}-- the same method is also used to justify rigorously
the estimate reported in Theorem \ref{Main thm }, and is therefore
not repeated here. The verification that the method of Section 4 (which,
as remarked, borrows from {[}15{]}) can be adapted to the forthcoming
calculations is standard.

\medskip{}

\textit{Proof of Theorem \ref{Main theorem 2}. }\textbf{Step 0}.
By Lemma (\ref{coupling propagates}) and Lemma (\ref{Equivalance })
for each $t\geq0$

\[
\widetilde{\mathcal{E}_{\hbar}^{\lambda}}(t)\geq\frac{1}{\max(2,\frac{\lambda^{2}+2K^{2}}{\lambda^{2}})}E_{\hbar}^{\lambda}(f(t),R_{\hbar}(t))^{2}.
\]
Set

\[
I_{1}\coloneqq\frac{i}{\hbar}\left[\frac{1}{2}\left|-i\hbar\nabla_{y}+A(y)\right|^{2},\widetilde{c}_{\hbar}^{\lambda}(x,\xi)\right]+\left\{ \frac{1}{2}\left|\xi+A(x)\right|^{2},\widetilde{c}_{\hbar}^{\lambda}(x,\xi)\right\} 
\]
and 

\[
I_{2}\coloneqq\frac{i}{\hbar}\left[\frac{1}{2}\left|y\right|^{2}+V(y),\widetilde{c}_{\hbar}^{\lambda}(x,\xi)\right]+\left\{ \frac{1}{2}\left|x\right|^{2}+V(x),\widetilde{c}_{\hbar}^{\lambda}(x,\xi)\right\} .
\]
We compute 

\[
\frac{d}{dt}\mathrm{trace}\left(\widetilde{c}_{\hbar}^{\lambda}(x,\xi)Q(t,x,\xi)\right)=\mathrm{trace}\left(\widetilde{c}_{\hbar}^{\lambda}(x,\xi))\partial_{t}Q(t,x,\xi)\right)
\]

\[
=-\mathrm{trace}\left(\widetilde{c}_{\hbar}^{\lambda}(x,\xi)\left\{ \frac{1}{2}\left|\xi+A(x)\right|^{2}+\frac{1}{2}\left|x\right|^{2}+V(x),Q\right\} \right)
\]

\[
-\mathrm{trace}\left(\widetilde{c}_{\hbar}^{\lambda}(x,\xi)\frac{i}{\hbar}\left[\frac{1}{2}\left|-i\hbar\nabla_{y}+A(y)\right|^{2}+\frac{1}{2}\left|y\right|^{2}+V(y),Q\right]\right)
\]

\[
=\mathrm{trace}\left(\left\{ \frac{1}{2}\left|\xi+A(x)\right|^{2}+\frac{1}{2}\left|x\right|^{2}+V(x),\widetilde{c}_{\hbar}^{\lambda}(x,\xi)\right\} Q(t,x,\xi)\right)
\]

\[
+\mathrm{trace}\left(\frac{i}{\hbar}\left[\frac{1}{2}\left|-i\hbar\nabla_{y}+A(y)\right|^{2}+\frac{1}{2}\left|y\right|^{2}+V(y),\widetilde{c}_{\hbar}^{\lambda}(x,\xi)\right]Q(t,x,\xi)\right)
\]

\[
=\mathrm{trace}\left(I_{1}Q(t,x,\xi)\right)+\mathrm{trace}\left(I_{2}Q(t,x,\xi)\right).
\]

\textbf{Step }1. \textbf{Estimating $\mathrm{trace}\left(I_{1}Q(t,x,\xi)\right)$.
}To make the equations a bit lighter, set 
\[
\Pi_{k}\coloneqq-i\hbar\partial_{k}+A_{k}(y)
\]
and

\[
D_{k}=D_{k}(x,\xi)\coloneqq\xi_{k}+A_{k}(x)-\Pi_{k}.
\]
In the following calculations Einstein summation is in free use. 
\[
\frac{i}{\hbar}\left[\frac{1}{2}\left|-i\hbar\nabla_{y}+A(y)\right|^{2},\widetilde{c}_{\hbar}^{\lambda}(x,\xi)\right]
\]

\[
=\frac{1}{2}\frac{i}{\hbar}\Pi_{k}\vee\left[\Pi_{k},\frac{\lambda^{2}}{2}(x_{l}-y_{l})^{2}+\frac{1}{2}D_{l}^{2}\right]
\]

\[
=\frac{\lambda^{2}}{4}\frac{i}{\hbar}\Pi_{k}\vee\left((x_{l}-y_{l})\vee\left[\Pi_{k},(x_{l}-y_{l})\right]\right)
\]

\[
+\frac{1}{2}\frac{i}{\hbar}\Pi_{k}\vee\left(D_{l}\vee\frac{1}{2}\left[\Pi_{k},D_{l}\right]\right)
\]

\begin{equation}
=-\frac{\lambda^{2}}{2}\Pi_{k}\lor((x_{l}-y_{l})\partial_{y_{k}}(y_{l}))+\frac{1}{4}\Pi_{k}\vee(D_{l}\vee\mathbf{a}_{kl}(y)),\label{eq:-7-2}
\end{equation}
where in the last equation we abbreviated $\mathbf{a}_{kl}(y)\coloneqq\partial_{y_{l}}A_{k}(y)-\partial_{y_{k}}A_{l}(y)$.
In addition 

\[
\left\{ \frac{1}{2}\left|\xi+A(x)\right|^{2},\widetilde{c}_{\hbar}^{\lambda}(x,\xi))\right\} =\left\{ \frac{1}{2}(\xi_{k}+A_{k}(x))^{2},\frac{\lambda^{2}}{2}(x_{l}-y_{l})^{2}+\frac{1}{2}D_{l}^{2}\right\} 
\]

\begin{equation}
=\lambda^{2}(\xi_{k}+A_{k}(x))(x_{k}-y_{k})-\frac{1}{4}\mathbf{a}_{kl}(x)\vee\left((\xi_{k}+A_{k}(x))\vee D_{l}\right).\label{eq:-7}
\end{equation}
Adding up the first terms of equations (\ref{eq:-7-2}) and (\ref{eq:-7})
gives 
\[
\mathrm{trace}\left(\left(\frac{\lambda^{2}}{2}(\xi_{k}+A_{k}(x))\vee(x_{k}-y_{k})-\frac{\lambda^{2}}{2}\Pi_{k}\lor((x_{l}-y_{l})\partial_{y_{k}}(y_{l}))\right)Q(t,x,\xi)\right)
\]

\[
=\mathrm{trace}\left(\frac{\lambda^{2}}{2}D_{k}\lor(x_{k}-y_{k})Q(t,x,\xi)\right)
\]

\begin{equation}
\leq\mathrm{trace}\left(\left(\frac{1}{2}D_{k}^{2}+\frac{\lambda^{2}}{2}(x_{k}-y_{k})^{2}\right)Q(t,x,\xi)\right).\label{eq:first inequality for I1}
\end{equation}
Adding up the second terms of equations (\ref{eq:-7-2}) and (\ref{eq:-7})
gives 
\[
\frac{1}{4}\Pi_{k}\vee(D_{l}\vee\mathbf{a}_{kl}(y))-\frac{1}{4}(\xi_{k}+A_{k}(x))\vee\left(D_{l}\vee\mathbf{a}_{kl}(x)\right)
\]

\[
=-\frac{1}{4}(\xi_{k}+A_{k}(x))\vee\left(D_{l}\vee(\mathbf{a}_{kl}(x)-\mathbf{a}_{kl}(y))\right)
\]

\[
-\frac{1}{4}(\xi_{k}+A_{k}(x))\vee\left(D_{l}\vee\mathbf{a}_{kl}(y)\right)+\frac{1}{4}\Pi_{k}\vee\left(D_{l}\vee\mathbf{a}_{kl}(y)\right)
\]

\[
=-\frac{1}{4}(\xi_{k}+A_{k}(x))\vee\left(D_{l}\vee(\mathbf{a}_{kl}(x)-\mathbf{a}_{kl}(y))\right)
\]

\[
-\frac{1}{4}D_{k}\vee\left(D_{l}\vee\mathbf{a}_{kl}(y)\right).
\]

We claim that the second line is identically $0$, i.e. 
\begin{claim}
\label{claim:.}$\underset{k,l}{\sum}D_{k}\vee(D_{l}\vee\mathbf{a}_{kl})=0$. 
\end{claim}
The proof of Claim (\ref{claim:.}) is postponed to the end of this
section. We arrive at 
\[
\mathrm{trace}\left(I_{1}Q(t,x,\xi)\right)\leq\mathrm{trace}\left(\widetilde{c}_{\hbar}^{\lambda}(\Phi_{t}(x,\xi))Q(t,x,\xi)\right)
\]

\begin{equation}
+\frac{1}{4}\mathrm{trace}\left((\xi_{k}+A_{k}(x))\vee\left(D_{l}\vee(\mathbf{a}_{kl}(y)-\mathbf{a}_{kl}(x))\right)Q(t,x,\xi)\right).\label{eq:inequality for I1}
\end{equation}
\textbf{Step }2. \textbf{Estimating $I_{2}$. }We compute 

\[
\frac{i}{\hbar}\left[V(y),\widetilde{c}_{\hbar}^{\lambda}(x,\xi)\right]=\frac{i}{\hbar}\left[V(y),\frac{\lambda^{2}}{2}\left|x-y\right|^{2}+\frac{1}{2}D_{k}^{2}\right]
\]

\[
=\frac{i}{\hbar}D_{k}\lor\left[V(y),D_{k}\right]
\]

\[
=\frac{i}{\hbar}D_{k}\lor\left[V(y),i\hbar\partial_{y_{k}}\right]
\]

\[
=\frac{1}{2}D_{k}\lor(\partial_{y_{k}}V(y)).
\]
The same calculation with $V$ replaced by $\frac{1}{2}\left|y\right|^{2}$
gives

\[
\frac{i}{\hbar}\left[\frac{1}{2}\left|y\right|^{2},\widetilde{c}_{\hbar}^{\lambda}(x,\xi)\right]=\frac{1}{2}D_{l}\vee y_{l}.
\]
The Poisson brackets are recast as 

\[
\left\{ \frac{1}{2}\left|x\right|^{2}+V(x),\widetilde{c}_{\hbar}^{\lambda}(x,\xi)\right\} =-(x_{k}+\partial_{x_{k}}V(x))D_{k}.
\]
Thus 

\[
I_{2}=\frac{1}{2}D_{k}\lor(\partial_{y_{k}}V(y))-\partial_{x_{k}}V(x)D_{k}
\]

\[
+\frac{1}{2}D_{l}\vee(y_{k}\partial_{y_{l}}y_{k})-\frac{1}{2}D_{l}\vee x_{l}
\]

\[
=\frac{1}{2}D_{k}\vee(\partial_{y_{k}}V(y)-\partial_{x_{k}}V(x))+\frac{1}{2}D_{k}\vee(y_{l}-x_{l}).
\]
Hence 

\[
\mathrm{trace}\left(I_{2}Q(t,x,\xi)\right)
\]

\[
\leq\mathrm{trace}\left(\left(\frac{1}{2}\left|\xi+i\hbar\nabla_{y}+A(x)-A(y)\right|^{2}+\frac{1}{2}L^{2}\left|x-y\right|^{2}\right)Q(t,x,\xi)\right)
\]

\[
+\mathrm{trace}\left(\left(\frac{1}{2}\left|\xi+i\hbar\nabla_{y}+A(x)-A(y)\right|^{2}+\frac{1}{2}\left|x-y\right|^{2}\right)Q(t,x,\xi)\right)
\]

\begin{equation}
\leq\max\left(1,\frac{1+L^{2}}{\lambda^{2}}\right)\mathrm{trace}\left(\widetilde{c}_{\hbar}^{\lambda}(x,\xi)Q(t,x,\xi)\right).\label{eq:inequality for I2}
\end{equation}
\textbf{Step 3}.\textbf{ Gronwall inequality and conclusion. }By inequalities
(\ref{eq:inequality for I1}) and (\ref{eq:inequality for I2}) we
have 
\[
\mathrm{trace}\left(\widetilde{c}_{\hbar}^{\lambda}(x,\xi)Q(t,x,\xi)\right)
\]

\[
\leq\mathrm{trace}\left(\widetilde{c}_{\hbar}^{\lambda}(x,\xi)Q^{\mathrm{in}}(x,\xi)\right)
\]

\[
+\frac{1}{4}\stackrel[0]{t}{\int}\mathrm{trace}\left((\xi_{k}+A_{k}(x))\vee\left(D_{l}\vee(\mathbf{a}_{kl}(y)-\mathbf{a}_{kl}(x))\right)Q(s,x,\xi)\right)ds
\]

\[
+\stackrel[0]{t}{\int}\left(1+\max\left(1,\frac{1+L^{2}}{\lambda^{2}}\right)\right)\mathrm{trace}\left(\widetilde{c}_{\hbar}^{\lambda}(x,\xi)Q(s,x,\xi)\right)ds.
\]
Integrating on $\mathbb{R}^{d}\times\mathbb{R}^{d}$ yields 

\[
\underset{\mathbb{R}^{d}\times\mathbb{R}^{d}}{\int}\mathrm{trace}\left(\widetilde{c}_{\hbar}^{\lambda}(x,\xi)Q(t,x,\xi)\right)dxd\xi
\]
\[
\leq\underset{\mathbb{R}^{d}\times\mathbb{R}^{d}}{\int}\mathrm{trace}\left(\widetilde{c}_{\hbar}^{\lambda}(x,\xi)Q^{\mathrm{in}}(x,\xi)\right)dxd\xi
\]

\[
+\frac{1}{4}\stackrel[0]{t}{\int}\underset{\mathbb{R}^{d}\times\mathbb{R}^{d}}{\int}\mathrm{trace}\left((\xi_{k}+A_{k}(x))\vee\left(D_{l}\vee(\mathbf{a}_{kl}(y)-\mathbf{a}_{kl}(x))\right)Q(s,x,\xi)\right)dxd\xi ds
\]

\begin{equation}
+\stackrel[0]{t}{\int}\underset{\mathbb{R}^{d}\times\mathbb{R}^{d}}{\int}\left(1+\max\left(1,\frac{1+L^{2}}{\lambda^{2}}\right)\right)\mathrm{trace}\left(\widetilde{c}_{\hbar}^{\lambda}(x,\xi)Q(s,x,\xi)\right)dxd\xi ds.\label{eq:prefinal inequality}
\end{equation}
The final ingredient needed in order to complete the estimate is mastering
the term on the second line, which is precisely the place where we
use the assumption that the support of $f^{\mathrm{in}}$ is ``not
too large''. 
\begin{claim}
\label{time depnedent support } The following estimate holds 

\[
\frac{1}{4}\stackrel[0]{t}{\int}\underset{\mathbb{R}^{d}\times\mathbb{R}^{d}}{\int}\mathrm{trace}\left((\xi_{k}+A_{k}(x))\vee\left(D_{l}\vee(\mathbf{a}_{kl}(y)-\mathbf{a}_{kl}(x))\right)Q(s,x,\xi)\right)dxd\xi ds
\]

\[
\leq4\rho_{0}\max\left(\frac{2K'^{2}}{\lambda^{2}},1\right)\underset{\mathbb{R}^{d}\times\mathbb{R}^{d}}{\int}\mathrm{trace}\left(\widetilde{c}_{\hbar}^{\lambda}(x,\xi)Q(t,x,\xi)\right)dxd\xi.
\]
\end{claim}
\begin{proof}
Let $\chi\in C_{0}^{\infty}(\mathbb{R})$ such that $0\leq\chi\leq1$,
$\chi(r)\equiv1$ for $\left|r\right|>2\rho_{0}^{2}$ and $\chi(r)\equiv0$
for $|r|<\rho_{0}^{2}$. Observe that 

\[
\frac{d}{dt}\underset{\mathbb{R}^{d}\times\mathbb{R}^{d}}{\int}\chi\left(\frac{1}{2}\left|\xi+A(x)\right|^{2}+\frac{1}{2}\left|x\right|^{2}+V(x)\right)f(t,x,\xi)dxd\xi
\]

\[
=-\underset{\mathbb{R}^{d}\times\mathbb{R}^{d}}{\int}\chi\left(\frac{1}{2}\left|\xi+A(x)\right|^{2}+\frac{1}{2}\left|x\right|^{2}+V(x)\right)\left\{ \frac{1}{2}\left|\xi+A(x)\right|^{2}+\frac{1}{2}\left|x\right|^{2}+V(x),f(t,x,\xi)\right\} dxd\xi
\]

\[
=\underset{\mathbb{R}^{d}\times\mathbb{R}^{d}}{\int}\left\{ \frac{1}{2}\left|\xi+A(x)\right|^{2}+\frac{1}{2}\left|x\right|^{2}+V(x),\chi\left(\frac{1}{2}\left|\xi+A(x)\right|^{2}+\frac{1}{2}\left|x\right|^{2}+V(x)\right)\right\} f(t,x,\xi)dxd\xi=0,
\]
and as a result 
\[
\underset{\mathbb{R}^{d}\times\mathbb{R}^{d}}{\int}\chi\left(\frac{1}{2}\left|\xi+A(x)\right|^{2}+\frac{1}{2}\left|x\right|^{2}+V(x)\right)f(t,x,\xi)dxd\xi
\]

\[
=\underset{\mathbb{R}^{d}\times\mathbb{R}^{d}}{\int}\chi\left(\frac{1}{2}\left|\xi+A(x)\right|^{2}+\frac{1}{2}\left|x\right|^{2}+V(x)\right)f^{\mathrm{in}}(x,\xi)dxd\xi=0.
\]
Since $Q(t,x,\xi)$ is a coupling of $f(t,x,\xi)$ and $R(t)$, it
follows that
\[
\chi\left(\frac{1}{2}\left|\xi+A(x)\right|^{2}+\frac{1}{2}\left|x\right|^{2}+V(x)\right)\mathrm{trace}\left(Q(t,x,\xi)\right)\equiv0,
\]
so that 

\[
\chi\left(\frac{1}{2}\left|\xi+A(x)\right|^{2}+\frac{1}{2}\left|x\right|^{2}+V(x)\right)Q(t,x,\xi)\equiv0.
\]
We thus conclude 

\[
\frac{1}{4}\underset{\mathbb{R}^{d}\times\mathbb{R}^{d}}{\int}\mathrm{trace}\left((\xi_{k}+A_{k}(x))\vee\left(D_{l}\vee(\mathbf{a}_{kl}(y)-\mathbf{a}_{kl}(x))\right)Q(t,x,\xi)\right)dxd\xi
\]

\[
=\frac{1}{2}\underset{\mathbb{R}^{d}\times\mathbb{R}^{d}}{\int}(1-\chi(\frac{1}{2}\left|\xi+A(x)\right|^{2}+\frac{1}{2}\left|x\right|^{2}+V(x)))(\xi_{k}+A_{k}(x))
\]

\[
\times\mathrm{trace}\left(D_{l}\vee(\mathbf{a}_{kl}(y)-\mathbf{a}_{kl}(x))Q(t,x,\xi)\right)dxd\xi
\]

\[
\leq2\rho_{0}\underset{\mathbb{R}^{d}\times\mathbb{R}^{d}}{\int}\left|\mathrm{trace}\left(D_{l}\vee(\mathbf{a}_{kl}(y)-\mathbf{a}_{kl}(x))Q(t,x,\xi)\right)\right|dxd\xi
\]

\[
\leq2\rho_{0}\underset{\mathbb{R}^{d}\times\mathbb{R}^{d}}{\int}\mathrm{trace}\left(\left(D_{l}^{2}+\left|\mathbf{a}_{kl}(y)-\mathbf{a}_{kl}(x)\right|^{2}\right)Q(t,x,\xi)\right)dxd\xi
\]

\[
\leq2\rho_{0}\underset{\mathbb{R}^{d}\times\mathbb{R}^{d}}{\int}\mathrm{trace}\left(\left(D_{l}^{2}+2K'^{2}\left|x-y\right|^{2}\right)Q(t,x,\xi)\right)dxd\xi
\]

\[
\leq4\rho_{0}\max\left(\frac{2K'^{2}}{\lambda^{2}},1\right)\underset{\mathbb{R}^{d}\times\mathbb{R}^{d}}{\int}\mathrm{trace}\left(\widetilde{c}_{\hbar}^{\lambda}(x,\xi)Q(t,x,\xi)\right)dxd\xi.
\]
The combination of inequality (\ref{eq:prefinal inequality}) and
Claim (\ref{time depnedent support }) yields 
\[
\underset{\mathbb{R}^{d}\times\mathbb{R}^{d}}{\int}\mathrm{trace}\left(\widetilde{c}_{\hbar}^{\lambda}(x,\xi)Q(t,x,\xi)\right)dxd\xi
\]
\[
\leq\underset{\mathbb{R}^{d}\times\mathbb{R}^{d}}{\int}\mathrm{trace}\left(\widetilde{c}_{\hbar}^{\lambda}(x,\xi)Q^{in}(x,\xi)\right)dxd\xi
\]
\end{proof}
\begin{equation}
+\stackrel[0]{t}{\int}\underset{\mathbb{R}^{d}\times\mathbb{R}^{d}}{\int}\left(1+\max\left(1,\frac{1+L^{2}}{\lambda^{2}}\right)+4\rho_{0}\max\left(\frac{2K'^{2}}{\lambda^{2}},1\right)\right)\mathrm{trace}\left(\widetilde{c}_{\hbar}^{\lambda}(x,\xi)Q(s,x,\xi)\right)dxd\xi ds.\label{eq:-5}
\end{equation}
Gronwall inequality as applied to inequality (\ref{eq:-5}) implies 

\[
\underset{\mathbb{R}^{d}\times\mathbb{R}^{d}}{\int}\mathrm{trace}\left(\widetilde{c}_{\hbar}^{\lambda}(x,\xi)Q(t,x,\xi)\right)dxd\xi\leq e^{\left(1+\max\left(1,\frac{1+L^{2}}{\lambda^{2}}\right)+4\rho_{0}\max\left(\frac{2K'^{2}}{\lambda^{2}},1\right)\right)t}\underset{\mathbb{R}^{d}\times\mathbb{R}^{d}}{\int}\mathrm{trace}\left(\widetilde{c}_{\hbar}^{\lambda}(x,\xi)Q^{in}(x,\xi)\right)dxd\xi.
\]
The application of Lemma (\ref{Equivalance }) produces

\[
\frac{1}{\max\left(2,\frac{\lambda^{2}+2K^{2}}{\lambda^{2}}\right)}E_{\hbar}^{\lambda}(f(t),R_{\hbar}(t))^{2}\leq\frac{1}{\max\left(2,\frac{\lambda^{2}+2K^{2}}{\lambda^{2}}\right)}\mathcal{E}_{\hbar}^{\lambda}(t)\leq\widetilde{\mathcal{E}_{\hbar}^{\lambda}}(t)
\]

\[
\leq\widetilde{\mathcal{E}_{\hbar}^{\lambda}}(0)e^{\alpha(K',L,\lambda,\rho_{0})t}\leq\max\left(2,\frac{\lambda^{2}+2K^{2}}{\lambda^{2}}\right)e^{\alpha(K',L,\lambda,\rho_{0})t}\mathcal{E}_{\hbar}^{\lambda}(0).
\]

Minimizing the right hand side of the above inequality as $Q^{\mathrm{in}}\in\mathcal{C}(f^{\mathrm{in}},R_{\hbar}^{\mathrm{in}})\cap\mathcal{D}^{2}(\mathfrak{H})$
yields 

\[
E_{\hbar}^{\lambda}(f(t),R_{\hbar}(t))^{2}\leq\beta(K,\lambda)e^{\alpha(K',L,\lambda,\rho_{0})t}E_{\hbar}^{\lambda}(f^{\mathrm{in}},R_{\hbar}^{\mathrm{in}})^{2},
\]

as claimed. 
\begin{onehalfspace}
\begin{flushright}
$\square$
\par\end{flushright}
\end{onehalfspace}

In order to finish the proof we are left to justify Claim (\ref{claim:.})
\begin{proof}
\textbf{Step 1.} \textbf{Calculation of $\left[D_{k},D_{l}\right]$}.
We expand

\[
D_{k}D_{l}=(\xi_{k}+A_{k}(x)-A_{k}(y)+i\hbar\partial_{y_{k}})(\xi_{l}+A_{l}(x)-A_{l}(y)+i\hbar\partial_{y_{l}})
\]

\[
=(\xi_{k}+A_{k}(x))(\xi_{l}+A_{l}(x))+(\xi_{k}+A_{k}(x))(-A_{l}(y)+i\hbar\partial_{y_{l}})+(-A_{k}(y)+i\hbar\partial_{y_{k}})(\xi_{l}+A_{l}(x))
\]

\[
-\hbar^{2}\partial_{y_{k}y_{l}}-i\hbar\partial_{y_{k}}A_{l}-A_{l}i\hbar\partial_{y_{k}}-A_{k}i\hbar\partial_{y_{l}}+A_{k}A_{l}.
\]
From this identity we derive the relation
\[
D_{k}D_{l}-D_{l}D_{k}=i\hbar\partial_{y_{l}}A_{k}-i\hbar\partial_{y_{k}}A_{l},
\]
so that 

\[
\mathbf{a}_{kl}(y)=\frac{1}{i\hbar}(D_{k}D_{l}-D_{l}D_{k}).
\]

\textbf{Step 2. .} \textbf{Calculation of $D_{l}\vee\mathbf{a}_{kl}$}.\textbf{
}By step 1 
\[
i\hbar D_{l}\vee\mathbf{a}_{kl}=D_{l}\vee(D_{k}D_{l}-D_{l}D_{k})
\]

\[
=D_{l}(D_{k}D_{l})+(D_{k}D_{l})D_{l}-D_{l}(D_{l}D_{k})-(D_{l}D_{k})D_{l}
\]

\[
=(D_{l}D_{k})D_{l}+(D_{k}D_{l})D_{l}-D_{l}(D_{l}D_{k})-(D_{l}D_{k})D_{l}
\]

\[
=(D_{k}D_{l})D_{l}-D_{l}(D_{l}D_{k}).
\]

\textbf{Step 3.} \textbf{Conclusion}. Thanks to step 2 we get \textbf{
\[
\underset{k,l}{\sum}D_{k}\vee(D_{l}\vee\mathbf{a}_{kl})=\frac{1}{i\hbar}\underset{k,l}{\sum}D_{k}\vee((D_{k}D_{l})D_{l}-D_{l}(D_{l}D_{k}))
\]
}

\[
=\frac{1}{i\hbar}\underset{k,l}{\sum}D_{k}((D_{k}D_{l})D_{l})+((D_{k}D_{l})D_{l})D_{k}-D_{k}(D_{l}(D_{l}D_{k}))-(D_{l}(D_{l}D_{k}))D_{k}
\]

\[
=\frac{1}{i\hbar}\underset{k,l}{\sum}((D_{k}D_{l})D_{l})D_{k}-D_{k}(D_{l}(D_{l}D_{k}))
\]

\begin{equation}
+\frac{1}{i\hbar}\underset{k,l}{\sum}D_{k}((D_{k}D_{l})D_{l})-\frac{1}{i\hbar}\underset{k,l}{\sum}(D_{l}(D_{l}D_{k}))D_{k}.\label{eq:-8}
\end{equation}
The first sum in the right handside of equation (\ref{eq:-8}) is
$0$ simply by associativity, while the two last sums cancel each
other by changing order of summation. 
\end{proof}
\begin{onehalfspace}
\begin{flushright}
\par\end{flushright}
\end{onehalfspace}

\section{\label{sec:3 OF CHAP 3} Estimate for $E_{\hbar,\epsilon}(f_{\epsilon}(t),R_{\epsilon,\hbar}(t))$}

In contrast with the previous section we consider here a \textit{double}
semi-classical limit as $\epsilon+\hbar\rightarrow0$, for which we
use the \textit{non} magnetic cost function with an $\epsilon^{2}$
weight in front of the quantum part. We prove here the following intermediate
inequality
\begin{thm}
\begin{flushleft}
\label{Pseudo distance estimates } Let $A(x)=\frac{1}{\epsilon}x^{\bot}$
and let $V$ satisfy (\ref{condition on V}). Let $f_{\epsilon}^{\mathrm{in}}$
be a probability density on $\mathbb{R}^{2}\times\mathbb{R}^{2}$
such that 
\par\end{flushleft}
\begin{flushleft}
\[
\underset{\mathbb{R}^{2}\times\mathbb{R}^{2}}{\int}\left(\left|x\right|^{2}+\left|\xi\right|^{2}\right)f_{\epsilon}^{\mathrm{in}}(x,\xi)dxd\xi<\infty.
\]
\par\end{flushleft}
\begin{flushleft}
Let $f_{\epsilon}$ be the solution of the Liouville equation (\ref{Vlasov})
with initial data $f_{\epsilon}^{\mathrm{in}}$. Let $R_{\hbar,\epsilon}^{\mathrm{in}}\in\mathcal{D}^{2}(\mathfrak{H})$
and let $R_{\epsilon,\hbar}$ be the solution to the von Neumann equation
(\ref{Hartree}) with initial data $R_{\epsilon,\hbar}^{\mathrm{in}}$.
Then 
\[
E_{\hbar,\epsilon}(f_{\epsilon}(t),R_{\hbar,\epsilon}(t))\leq e^{\max(2,\epsilon^{2}(1+L^{2}))t}E_{\hbar,\epsilon}(f_{\epsilon}^{\mathrm{in}},R_{\hbar,\epsilon}^{\mathrm{in}}).
\]
\par\end{flushleft}

\end{thm}
For each $Q^{\mathrm{in}}(x,\xi)\in\mathcal{C}(f_{\epsilon}^{\mathrm{in}},R_{\hbar,\epsilon}^{\mathrm{in}})\cap\mathcal{D}^{2}(\mathfrak{H})$
let $Q=Q_{\hbar,\epsilon}(t,x,\xi)$ be defined by 

\begin{equation}
Q(t,x,\xi)\coloneqq U^{\ast}(t)Q^{\mathrm{in}}(\Phi_{-t}^{\epsilon}(x,\xi))U(t).\label{eq:-28}
\end{equation}
We consider the time dependent quantity 

\[
\mathcal{E}_{\hbar,\epsilon}(t)\coloneqq\underset{\mathbb{R}^{d}\times\mathbb{R}^{d}}{\int}\mathrm{trace}\left(\sqrt{Q(t,x,\xi)}\left(\frac{1}{2}\left|x-y\right|^{2}+\frac{1}{2}\epsilon^{2}\left|\xi+i\hbar\nabla_{y}\right|^{2}\right)\sqrt{Q(t,x,\xi)}\right)dxd\xi\geq0.
\]
By Lemma \ref{finite second quantum moments } and Lemma \ref{finite second moments },
$R(t)\in\mathcal{D}^{2}(\mathfrak{H})$ and $f(t)$ has finite second
moments, which in view of Lemma \ref{Trace estimate} justifies that
$\mathcal{E}_{\hbar,\epsilon}(t)<\infty$. 

\medskip{}

\textit{Proof of theorem (\ref{Main theorem 2})}. \textbf{Step 0.
Smooth Approximation.} By Lemma \ref{coupling propagates} 

\[
\mathcal{E}_{\hbar,\epsilon}(t)\geq E_{\hbar,\epsilon}(f_{\epsilon}(t),R_{\epsilon,\hbar}(t))^{2}.
\]
In addition 

\[
\mathcal{E}_{\hbar,\epsilon}(t)=\underset{\mathbb{R}^{d}\times\mathbb{R}^{d}}{\int}\mathrm{trace}\left(U^{\ast}(t)\sqrt{Q^{in}(x,\xi)}U^{\ast}(t)c_{\hbar,\epsilon}(\Phi_{t}(x,\xi))U(t)\sqrt{Q^{in}(x,\xi)}U(t)\right)dxd\xi.
\]
For a.e. $(x,\xi)\in\mathbb{R}^{d}\times\mathbb{R}^{d}$ let $e_{1}(x,\xi),...,e_{k}(x,\xi),...$
be a $\mathfrak{H}$--complete orthonormal system of eigenvectors
of $Q^{\mathrm{in}}(x,\xi)$ with eigenvalues $\mu_{1}(x,\xi),...,\mu_{k}(x,\xi),...$
respectively. Then 

\[
\mathrm{trace}\left(U^{\ast}(t)\sqrt{Q^{in}(x,\xi)}U(t)c_{\hbar,\epsilon}(\Phi_{t}(x,\xi))U^{\ast}(t)\sqrt{Q^{in}(x,\xi)}U(t)\right)
\]

\begin{equation}
=\stackrel[k=1]{\infty}{\sum}\mu_{k}\left\langle U^{\ast}(t)e_{k},c_{\hbar,\epsilon}(\Phi_{t}(x,\xi))U^{\ast}(t)e_{k}\right\rangle .\label{eq:-4-3}
\end{equation}
For each $\varphi\in C_{0}^{\infty}(\mathbb{R}^{d})$ the map 

\[
t\mapsto\left\langle U^{\ast}(t)\varphi,c_{\hbar,\epsilon}(\Phi_{t}(x,\xi))U^{\ast}(t)\varphi\right\rangle 
\]

\begin{quotation}
is $\mathrm{Lip}([0,\tau])$ (see Lemma \ref{time derivative } for
more details) and one computes its time derivative as follows 
\end{quotation}
\[
\frac{d}{dt}\left\langle U^{\ast}(t)\varphi,c_{\hbar,\epsilon}(\Phi_{t}(x,\xi))U^{\ast}(t)\varphi\right\rangle 
\]

\[
=\left\langle -\frac{i}{\hbar}\mathscr{H}U^{\ast}(t)\varphi,c_{\hbar,\epsilon}(\Phi_{t}(x,\xi))U^{\ast}(t)\varphi\right\rangle 
\]

\[
-\left\langle U^{\ast}(t)\varphi,c_{\hbar,\epsilon}(\Phi_{t}(x,\xi))\frac{i}{\hbar}\mathscr{H}U^{\ast}(t)\varphi\right\rangle 
\]

\[
+\left\langle U^{\ast}(t)\varphi,\frac{d}{dt}(c_{\hbar,\epsilon}(\Phi_{t}(x,\xi)))U^{\ast}(t)\varphi\right\rangle 
\]

\[
=\frac{i}{\hbar}\left\langle \mathscr{H}U^{\ast}(t)\varphi,c_{\hbar,\epsilon}(\Phi_{t}(x,\xi))U^{\ast}(t)\varphi\right\rangle 
\]

\[
-\frac{i}{\hbar}\left\langle U^{\ast}(t)\varphi,c_{\hbar,\epsilon}(\Phi_{t}(x,\xi))\mathscr{H}U^{\ast}(t)\varphi\right\rangle 
\]

\[
+\left\langle U^{\ast}(t)\varphi,((X(t)-y)\cdot(\Xi(t)+A(X(t))+\right.
\]

\[
\left.\frac{\epsilon^{2}}{2}(\Xi(t)+i\hbar\nabla_{y})\vee(-\nabla A(X(t))(\Xi+A(X(t)))-X(t)-\nabla V(X(t)))U^{\ast}(t)\varphi\right\rangle 
\]
The first two terms are paired together to yield a commutator while
the last term is recognized as a Poisson bracket: 

\[
\frac{d}{dt}\left\langle U^{\ast}(t)\varphi,c_{\hbar,\epsilon}(\Phi_{t}(x,\xi))U^{\ast}(t)\varphi\right\rangle 
\]

\[
=\frac{i}{\hbar}\left\langle U^{\ast}(t)\varphi,\left[\mathscr{H},c_{\hbar,\epsilon}(\Phi_{t}(x,\xi))\right]U^{\ast}(t)\varphi\right\rangle 
\]

\[
+\left\langle U^{\ast}(t)\varphi,\left\{ \frac{1}{2}\left|\Xi+A(X)\right|^{2}+\frac{1}{2}\left|X\right|^{2}+V(X),\frac{1}{2}\left|X-y\right|^{2}+\frac{\epsilon^{2}}{2}\left|\Xi+i\hbar\nabla_{y}\right|^{2}\right\} (\Phi_{t}(x,\xi))U^{\ast}(t)\varphi\right\rangle .
\]
Here the Poisson bracket is with respect to $X,\Xi$. We omit the
time variable, since it will be invisible in the forthcoming calculations. 

. 
\[
\frac{i}{\hbar}\left[\mathscr{H},c_{\hbar,\epsilon}(\Phi_{t}(x,\xi))\right]
\]

\[
+\left\{ \frac{1}{2}\left|\Xi+A(X)\right|^{2}+\frac{1}{2}\left|X\right|^{2}+V(X),\frac{1}{2}\left|X-y\right|^{2}+\frac{\epsilon^{2}}{2}\left|\Xi+i\hbar\nabla_{y}\right|^{2}\right\} (\Phi_{t}(x,\xi))
\]

\[
=\left\{ \frac{1}{2}\left|\Xi+A(X)\right|^{2},\frac{1}{2}\left|X-y\right|^{2}+\frac{\epsilon^{2}}{2}\left|\Xi+i\hbar\nabla_{y}\right|^{2}\right\} 
\]

\[
+\frac{i}{\hbar}\left[\frac{1}{2}\left|-i\hbar\nabla_{y}+A(y)\right|^{2},\frac{1}{2}\left|X-y\right|^{2}+\frac{\epsilon^{2}}{2}\left|\Xi+i\hbar\nabla_{y}\right|^{2}\right]
\]

\[
+\left\{ V(X)+\frac{1}{2}\left|X\right|^{2},\frac{1}{2}\left|X-y\right|^{2}+\frac{\epsilon^{2}}{2}\left|\Xi+i\hbar\nabla_{y}\right|^{2}\right\} 
\]

\[
+\frac{i}{\hbar}\left[V(y)+\frac{1}{2}\left|y\right|^{2},\frac{1}{2}\left|X-y\right|^{2}+\frac{\epsilon^{2}}{2}\left|\Xi+i\hbar\nabla_{y}\right|^{2}\right].
\]
Denote by $\delta_{1}(t)$ the sum of the first two terms and by $\delta_{2}(t)$
the sum of the last two terms. We proceed through the following steps.

\textbf{Step 1. Vanishing of $\delta_{1}(t)$}. The vanishing of $\delta_{1}(t)$
reflects the main novelty of this section, as it is the main reason
for the fact that the final estimate is uniform in $\epsilon$. The
estimate of $\delta_{2}(t)$ would follow by an argument similar (and
in fact simpler) to the one in {[}14{]}. Recall that for brevity
we denote 
\[
\Pi_{k}\coloneqq-i\hbar\partial_{k}+A_{k}(y)
\]
and

\[
D_{k}=D_{k}(X,\Xi)\coloneqq\Xi_{k}+A_{k}(X)-\Pi_{k}.
\]
As usual, Einstein summation is freely used. We expand 

\[
\frac{i}{\hbar}\left[\frac{1}{2}\left|-i\hbar\nabla_{y}+A(y)\right|^{2},\frac{1}{2}\left|X_{l}-y_{l}\right|^{2}\right]
\]

\[
=\frac{1}{4}\frac{i}{\hbar}\Pi_{k}\lor\left((X_{l}-y_{l})\lor\left[\Pi_{k},(X_{l}-y_{l})\right]\right)
\]

\[
=\frac{1}{4}\frac{i}{\hbar}\Pi_{k}\lor\left((X_{l}-y_{l})\lor(i\hbar\partial_{k}y_{l})\right)
\]

\[
=-\frac{1}{4}\Pi_{k}\lor\left((X_{l}-y_{l})\lor(\partial_{k}y_{l})\right)=-\frac{1}{2}\Pi_{k}\lor(X_{k}-y_{k}),
\]

where the second equation is because the commutator of two multiplication
operators is $0$. Furthermore 
\[
\frac{i}{\hbar}\left[\frac{1}{2}(-i\hbar\partial_{y_{k}}+A_{k}(y))^{2},\frac{\epsilon^{2}}{2}(\Xi_{l}+i\hbar\partial_{y_{l}})^{2}\right]
\]

\[
=\frac{\epsilon^{2}}{4}\frac{i}{\hbar}\left[\Pi_{k}^{2},(\Xi_{l}+i\hbar\partial_{y_{l}})^{2}\right]
\]

\[
=\frac{\epsilon^{2}}{4}\frac{i}{\hbar}\Pi_{k}\lor\left((\Xi_{l}+i\hbar\partial_{y_{l}})\lor\left[\Pi_{k},(\Xi_{l}+i\hbar\partial_{y_{l}})\right]\right)
\]

\[
=\frac{\epsilon^{2}}{4}\frac{i}{\hbar}\Pi_{k}\lor\left((\Xi_{l}+i\hbar\partial_{y_{l}})\lor(-i\hbar\partial_{y_{l}}A_{k}(y))\right)
\]

\[
=\frac{\epsilon^{2}}{4}\Pi_{k}\lor\left((\Xi_{l}+i\hbar\partial_{y_{l}})\lor(\partial_{y_{l}}A_{k}(y))\right)
\]

\[
=\frac{\epsilon^{2}}{2}\Pi_{k}\lor((\Xi_{l}+i\hbar\partial_{y_{l}})(\nabla A)_{lk}),
\]
so that 

\[
\frac{i}{\hbar}\left[\frac{1}{2}\left|-i\hbar\nabla_{y}+A(y)\right|^{2},\frac{1}{2}\left|X-y\right|^{2}+\frac{\epsilon^{2}}{2}\left|\Xi+i\hbar\nabla_{y}\right|^{2}\right]
\]

\begin{equation}
=-\frac{1}{2}\Pi_{k}\lor(X_{k}-y_{k})+\frac{\epsilon^{2}}{2}\Pi_{k}\lor((\Xi+i\hbar\nabla_{y})\cdot(\nabla A))_{k}.\label{eq:-2-1-1-1}
\end{equation}
As for the Poisson brackets, we compute that 

\[
\left\{ (\Xi_{k}+A_{k}(X)){}^{2},\epsilon^{2}(\Xi_{l}+i\hbar\partial_{y_{l}})^{2}\right\} 
\]

\[
=\epsilon^{2}(\Xi_{k}+A_{k}(X))\lor\left\{ \Xi_{k}+A_{k}(X),(\Xi_{l}+i\hbar\partial_{y_{l}})^{2}\right\} 
\]

\[
=-\epsilon^{2}(\Xi_{k}+A_{k}(X))\lor\left\{ (\Xi_{l}+i\hbar\partial_{y_{l}})^{2},\Xi_{k}+A_{k}(X)\right\} 
\]

\[
=-\epsilon^{2}(\Xi_{k}+A_{k}(X))\lor\left((\Xi_{l}+i\hbar\partial_{y_{l}})\lor\left\{ \Xi_{l}+i\hbar\partial_{y_{l}},\Xi_{k}+A_{k}(X)\right\} \right)
\]

\begin{equation}
=-\epsilon^{2}(\Xi_{k}+A_{k}(X))\lor((\Xi_{l}+i\hbar\partial_{y_{l}})\lor\partial_{X_{l}}A_{k}),\label{eq:-1-1-1}
\end{equation}
since 

\[
\left\{ \Xi_{l}+i\hbar\partial_{y_{l}},\Xi_{k}+A_{k}(X)\right\} =\left\{ \Xi_{l},A_{k}(X)\right\} =\partial_{x_{l}}A_{k}.
\]
In addition 
\begin{equation}
\left\{ \frac{1}{2}\left|\Xi+A(X)\right|^{2},\frac{1}{2}\left|X-y\right|^{2}\right\} =(\Xi+A(X))\cdot(X-y).\label{eq:-20-1-1}
\end{equation}
So, gathering equations (\ref{eq:-2-1-1-1}), (\ref{eq:-1-1-1}) and
(\ref{eq:-20-1-1}) we get 

\[
(\Xi+A(X))\cdot(X-y)-\frac{1}{4}\epsilon^{2}(\Xi_{k}+A_{k}(X))\lor((\Xi_{l}+i\hbar\partial_{y_{l}})\lor\partial_{X_{l}}(A_{k}))
\]

\[
-\frac{1}{2}\Pi_{k}\lor((X_{k}-y_{k}))+\frac{\epsilon^{2}}{2}\Pi_{k}\lor((\Xi+i\hbar\nabla_{y})\cdot(\nabla A))_{k}
\]

\[
=\frac{1}{2}D_{k}\vee(X_{k}-y_{k})-\frac{\epsilon^{2}}{2}D_{k}\vee((\Xi+i\hbar\nabla_{y})\cdot(\nabla A))_{k}
\]

\[
=\frac{1}{2}(i\hbar\partial_{y_{k}}+\Xi_{k})\vee(X_{k}-y_{k})-\frac{\epsilon^{2}}{2}D_{k}\vee((\Xi+i\hbar\nabla_{y})\cdot(\nabla A))_{k},
\]
where the last equation is because $A(y)-A(X)\bot X-y$. Since $\nabla A=\frac{1}{\epsilon}\mathbf{J}$
where $\mathbf{J}\coloneqq\begin{pmatrix}0 & 1\\
-1 & 0
\end{pmatrix}$ we see that 

\[
\frac{1}{2}(i\hbar\nabla_{y}+\Xi)\vee(X-y)+\frac{1}{2}\epsilon^{2}(A(y)-A(X))\vee((\Xi+i\hbar\nabla_{y})\cdot(\nabla A))
\]
\[
=\frac{1}{2}(i\hbar\nabla_{y}+\Xi)\vee(X-y)+\frac{1}{2}(y-X)^{\bot}\vee(\xi+i\hbar\nabla_{y})^{\bot}=0.
\]
Moreover 

\[
(i\hbar\partial_{y_{k}}+\Xi_{k})\vee((\Xi+i\hbar\nabla_{y})\cdot(\nabla A))_{k}=0,
\]
which shows that $\delta_{1}(t)=0$. 

\textbf{Step 2.} \textbf{Controlling $\delta_{2}(t)$}.We have 
\[
\left\{ V(X),\frac{1}{2}\left|X-y\right|^{2}+\frac{\epsilon^{2}}{2}\left|\Xi+i\hbar\nabla_{y}\right|^{2}\right\} 
\]

\[
=\left\{ V(X),\frac{1}{2}\epsilon^{2}\left|\Xi+i\hbar\nabla_{y}\right|^{2}\right\} =-\epsilon^{2}\nabla V(X)\cdot(\Xi+i\hbar\nabla_{y}),
\]
and 
\[
\frac{i}{\hbar}\left[V(y),\frac{1}{2}\left|X-y\right|^{2}+\frac{1}{2}\epsilon^{2}\left|\Xi+i\hbar\nabla_{y}\right|^{2}\right]=\frac{i\epsilon^{2}}{2\hbar}\left[V(y),\left|\Xi+i\hbar\nabla_{y}\right|^{2}\right]
\]

\[
=\frac{i\epsilon^{2}}{2\hbar}\left[V(y),(\Xi_{k}+i\hbar\partial_{y_{k}})^{2}\right]=\frac{i\epsilon^{2}}{2\hbar}(\Xi_{k}+i\hbar\partial_{y_{k}})\vee\left[V(y),\Xi_{k}+i\hbar\partial_{y_{k}}\right]
\]

\[
=-\frac{i\epsilon^{2}}{2\hbar}(\Xi_{k}+i\hbar\partial_{y_{k}})\vee i\hbar\partial_{y_{k}}V=\frac{\epsilon^{2}}{2}(\Xi_{k}+i\hbar\partial_{y_{k}})\vee\partial_{y_{k}}V.
\]
Therefore we may write 

\[
\left\{ V(X),\frac{1}{2}\left|X-y\right|^{2}+\frac{\epsilon^{2}}{2}\left|\Xi+i\hbar\nabla_{y}\right|^{2}\right\} +\frac{i}{\hbar}\left[V(y),\frac{1}{2}\left|X-y\right|^{2}+\frac{\epsilon^{2}}{2}\left|\Xi+i\hbar\nabla_{y}\right|^{2}\right]
\]

\[
=-\epsilon^{2}\partial_{X_{k}}V(X)(\Xi_{k}+i\hbar\partial_{y_{k}})+\frac{\epsilon^{2}}{2}(\Xi_{k}+i\hbar\partial_{y_{k}})\vee\partial_{y_{k}}V(y)
\]

\[
=\frac{\epsilon^{2}}{2}(\partial_{y_{k}}V(y)-\partial_{X_{k}}V(X))\lor(\Xi_{k}+i\hbar\partial_{y_{k}}).
\]
Furthermore 

\[
\left\{ \frac{1}{2}\left|X\right|^{2},\frac{1}{2}\left|X-y\right|^{2}+\frac{\epsilon^{2}}{2}\left|\Xi+i\hbar\nabla_{y}\right|^{2}\right\} +\frac{i}{\hbar}\left[\frac{1}{2}\left|y\right|^{2},\frac{1}{2}\left|X-y\right|^{2}+\frac{\epsilon^{2}}{2}\left|\Xi+i\hbar\nabla_{y}\right|^{2}\right]
\]

\[
=-\frac{\epsilon^{2}}{2}X_{k}\vee(\Xi_{k}+i\hbar\partial_{y_{k}})+\frac{\epsilon^{2}}{2}(\Xi_{k}+i\hbar\partial_{y_{k}})\vee y_{k}=\frac{\epsilon^{2}}{2}(\Xi_{k}+i\hbar\partial_{y_{k}})\vee(X_{k}-y_{k}).
\]
Hence 

\[
\left\langle U^{\ast}(t)\varphi,\delta_{2}(t)U^{\ast}(t)\varphi\right\rangle 
\]

\[
=\left\langle U^{\ast}(t)\varphi,\frac{\epsilon^{2}}{2}(\partial_{y_{k}}V(y)-\partial_{k}V(X))\lor(\Xi_{k}+i\hbar\partial_{y_{k}})U^{\ast}(t)\varphi\right\rangle 
\]

\[
+\left\langle U^{\ast}(t)\varphi,\frac{\epsilon^{2}}{2}(X_{k}-y_{k})\lor(\Xi_{k}+i\hbar\partial_{y_{k}})U^{\ast}(t)\varphi\right\rangle 
\]

\[
\leq\frac{\epsilon^{2}}{2}\left\langle U^{\ast}(t)\varphi,(L^{2}\left|X-y\right|^{2}+(\Xi_{k}+i\hbar\partial_{y_{k}})^{2})U^{\ast}(t)\varphi\right\rangle 
\]

\[
+\frac{\epsilon^{2}}{2}\left\langle U^{\ast}(t)\varphi,(\left|X-y\right|^{2}+(\Xi_{k}+i\hbar\partial_{y_{k}})^{2})U^{\ast}(t)\varphi\right\rangle 
\]

\begin{equation}
\leq\max(2,\epsilon^{2}(1+L^{2}))\left\langle U^{\ast}(t)\varphi,c_{\hbar,\epsilon}(\Phi_{t}(x,\xi))U^{\ast}(t)\varphi\right\rangle .\label{eq:-1-3}
\end{equation}

\textbf{Step 3. Gronwall estimate. }Combining step 1 with inequality
(\ref{eq:-1-3}) gives 

\[
\left\langle U^{\ast}(t)\varphi,c_{\hbar,\epsilon}(\Phi_{t}(x,\xi))U^{\ast}(t)\varphi\right\rangle 
\]

\[
\leq\left\langle \varphi,c_{\hbar,\epsilon}(x,\xi)\varphi\right\rangle +\max(2,\epsilon^{2}(1+L^{2}))\stackrel[0]{t}{\int}\left\langle U^{\ast}(s)\varphi,c_{\hbar,\epsilon}(\Phi_{s}(x,\xi))U^{\ast}(s)\varphi\right\rangle ds,
\]
which by Gronwall implies 

\[
\left\langle U^{\ast}(t)\varphi,c_{\hbar,\epsilon}(\Phi_{t}(x,\xi))U^{\ast}(t)\varphi\right\rangle \leq\left\langle \varphi,c_{\hbar,\epsilon}(x,\xi)\varphi\right\rangle e^{\max(2,\epsilon^{2}(1+L^{2}))t}.
\]
By density of $C_{0}^{\infty}(\mathbb{R}^{2})$ in the domain of $c_{\hbar,\epsilon}(x,\xi)$
this implies (see Lemma \ref{Smooth approx} for more details) 

\[
\left\langle U^{\ast}(t)e_{k},c_{\hbar,\epsilon}(\Phi_{t}(x,\xi))U^{\ast}(t)e_{k}\right\rangle \leq\left\langle e_{k},c_{\hbar,\epsilon}(x,\xi)e_{k}\right\rangle e^{\max(2,\epsilon^{2}(1+L^{2}))t}.
\]
Multiplying both sides by $\mu_{k}\geq0$ and summing over $k$ yields
(for a.e. $(x,\xi)$) 

\[
\stackrel[k=1]{\infty}{\sum}\mu_{k}\left\langle U^{\ast}(t)e_{k},c_{\hbar,\epsilon}(\Phi_{t}(x,\xi))U^{\ast}(t)e_{k}\right\rangle 
\]

\[
\leq\stackrel[k=1]{\infty}{\sum}\mu_{k}\left\langle e_{k},c_{\hbar,\epsilon}(x,\xi)e_{k}\right\rangle e^{\max(2,\epsilon^{2}(1+L^{2}))t}.
\]
Integrating both sides on $\mathbb{R}^{2}\times\mathbb{R}^{2}$ produces

\[
E_{\hbar,\epsilon}(f_{\epsilon}(t),R_{\hbar,\epsilon}(t))^{2}\leq\mathcal{E}_{\hbar,\epsilon}(t)\leq\mathcal{E}_{\hbar,\epsilon}(0)e^{\max(2,\epsilon^{2}(1+L^{2}))t}.
\]
Minimizing the right hand side of the above inequality as $Q^{in}\in\mathcal{C}(f_{\epsilon}^{\mathrm{in}},R_{\hbar,\epsilon}^{\mathrm{in}})\cap\mathcal{D}^{2}(\mathfrak{H})$
yields 

\[
E_{\hbar,\epsilon}(f_{\epsilon}(t),R_{\hbar,\epsilon}(t))^{2}\leq e^{\max(2,\epsilon^{2}(1+L^{2}))t}E_{\hbar,\epsilon}(f_{\epsilon}^{\mathrm{in}},R_{\hbar,\epsilon}^{\mathrm{in}})^{2},
\]
as claimed. 
\begin{onehalfspace}
\begin{flushright}
$\square$
\par\end{flushright}
\end{onehalfspace}

\medskip{}

The following lemmata elaborate on the approximation procedure used
in the proof of Theorems \ref{Main thm } and \ref{Main theorem 2}.
With the same notation as before we have
\begin{lem}
\label{time derivative } Let $\varphi\in C_{0}^{\infty}(\mathbb{R}^{d})$.
The map $t\mapsto\left\langle U^{\ast}(t)\varphi,c_{\hbar}^{\lambda}(\Phi_{t}(x,\xi))U^{\ast}(t)\varphi\right\rangle $belongs
to $\mathrm{Lip}([0,\tau])$. 
\end{lem}
\textit{Proof}. Set 
\[
\mathbf{c}^{1}(t,x,\xi)\coloneqq\lambda\left|X(t)-y\right|
\]
and 

\[
\mathbf{c}^{2}(t,x,\xi)\coloneqq\Xi(t)+i\hbar\nabla_{y}.
\]
Then

\[
\left|\left\langle U^{\ast}(t)\varphi,c_{\hbar}^{\lambda}(\Phi_{t}(x,\xi))U^{\ast}(t)\varphi\right\rangle -\left\langle U^{\ast}(s)\varphi,c_{\hbar}^{\lambda}(\Phi_{s}(x,\xi))U^{\ast}(s)\varphi\right\rangle \right|
\]
 
\[
\leq\frac{1}{2}\left|\left\langle \mathbf{c}^{1}(t,x,\xi)U^{\ast}(t)\varphi,\mathbf{c}^{1}(t,x,\xi)U^{\ast}(t)\varphi\right\rangle -\left\langle \mathbf{c}^{1}(s,x,\xi)U^{\ast}(s)\varphi,\mathbf{c}^{1}(s,x,\xi)U^{\ast}(s)\varphi\right\rangle \right|
\]

\[
+\frac{1}{2}\left|\left\langle \mathbf{c}^{2}(t,x,\xi)U^{\ast}(t)\varphi,\mathbf{c}^{2}(t,x,\xi)U^{\ast}(t)\varphi\right\rangle -\left\langle \mathbf{c}^{2}(s,x,\xi)U^{\ast}(s)\varphi,\mathbf{c}^{2}(s,x,\xi)U^{\ast}(s)\varphi\right\rangle \right|.
\]
By elementary properties of semigroups (see e.g. Section 7.4.1 in
{[}12{]}) there is a constant $M>0$ such that 
\[
\left|\left\langle \mathbf{c}^{1}(t,x,\xi)U^{\ast}(t)\varphi,\mathbf{c}^{1}(t,x,\xi)U^{\ast}(t)\varphi\right\rangle -\left\langle \mathbf{c}^{1}(s,x,\xi)U^{\ast}(s)\varphi,\mathbf{c}^{1}(s,x,\xi)U^{\ast}(s)\varphi\right\rangle \right|
\]

\[
\leq\left|\left\langle \mathbf{c}^{1}(t,x,\xi)U^{\ast}(t)\varphi,\mathbf{c}^{1}(t,x,\xi)U^{\ast}(t)\varphi\right\rangle -\left\langle \mathbf{c}^{1}(s,x,\xi)U^{\ast}(s)\varphi,\mathbf{c}^{1}(t,x,\xi)U^{\ast}(t)\varphi\right\rangle \right|
\]

\[
+\left|\left\langle \mathbf{c}^{1}(s,x,\xi)U^{\ast}(s)\varphi,\mathbf{c}^{1}(t,x,\xi)U^{\ast}(t)\varphi\right\rangle -\left\langle \mathbf{c}^{1}(s,x,\xi)U^{\ast}(s)\varphi,\mathbf{c}^{1}(t,x,\xi)U^{\ast}(s)\varphi\right\rangle \right|
\]

\begin{flushright}
\[
\leq M\left\Vert \mathbf{c}^{1}(t,x,\xi)U^{\ast}(t)\varphi\right\Vert _{L^{\infty}([0,\tau];\mathfrak{H})}\left|t-s\right|.
\]
\par\end{flushright}

Similarly 
\[
\left|\left\langle \mathbf{c}^{2}(t,x,\xi)U^{\ast}(t)\varphi,\mathbf{c}^{2}(t,x,\xi)U^{\ast}(t)\varphi\right\rangle -\left\langle \mathbf{c}^{2}(s,x,\xi)U^{\ast}(s)\varphi,\mathbf{c}^{2}(s,x,\xi)U^{\ast}(s)\varphi\right\rangle \right|
\]

\[
\leq M\left\Vert \mathbf{c}^{2}(t,x,\xi)U^{\ast}(t)\varphi\right\Vert _{L^{\infty}([0,\tau];\mathfrak{H})}\left|t-s\right|.
\]
By Corollary \ref{domain of magnetic vs non magnetic } in Section
\ref{domainS SECTION} we have $D(\mathscr{H})=D(\mathscr{O})$ (recall
the notation $\mathscr{O}\coloneqq-\frac{\hbar^{2}}{2}\Delta+\frac{1}{2}\left|y\right|^{2}$)
so that 
\[
U^{\ast}(t)\varphi\in C^{1}([0,\tau];D(\mathscr{O})).
\]
In particular, the mixed norms $\left\Vert \mathbf{c}^{1}(t,x,\xi)U^{\ast}(t)\varphi\right\Vert _{L^{\infty}([0,\tau];\mathfrak{H})},\left\Vert \mathbf{c}^{2}(t,x,\xi)U^{\ast}(t)\varphi\right\Vert _{L^{\infty}([0,\tau];\mathfrak{H})}$
are finite, which establishes the statement.
\begin{onehalfspace}
\begin{flushright}
$\square$
\par\end{flushright}
\end{onehalfspace}
\begin{lem}
\label{Smooth approx} Fix $k\in\mathbb{N}$ and $(x,\xi)\in\mathbb{R}^{d}\times\mathbb{R}^{d}$.
Let $\{\varphi_{n}\}_{n=1}^{\infty}\subset C_{0}^{\infty}(\mathbb{R}^{d})$
such that $\varphi_{n}\underset{n\rightarrow\infty}{\rightarrow}e_{k}$
in $D(\mathscr{H})$. Then 

\[
\left\langle U^{\ast}(t)\varphi_{n},c_{\hbar}^{\lambda}(x,\xi)U^{\ast}(t)\varphi_{n}\right\rangle \underset{n\rightarrow\infty}{\rightarrow}\left\langle U^{\ast}(t)e_{k},c_{\hbar}^{\lambda}(x,\xi)U^{\ast}(t)e_{k}\right\rangle .
\]
\end{lem}
\begin{proof}
One has

\[
\left|\left\langle U^{\ast}(t)\varphi_{n},c_{\hbar}^{\lambda}(x,\xi)U^{\ast}(t)\varphi_{n}\right\rangle -\left\langle U^{\ast}(t)e_{k},c_{\hbar}^{\lambda}(x,\xi)U^{\ast}(t)e_{k}\right\rangle \right|
\]

\[
\leq\left|\left\langle U^{\ast}(t)\varphi_{n},c_{\hbar}^{\lambda}(x,\xi)U^{\ast}(t)(\varphi_{n}-e_{k})\right\rangle \right|
\]

\[
+\left|\left\langle U^{\ast}(t)(\varphi_{n}-e_{k}),c_{\hbar}^{\lambda}(x,\xi)U^{\ast}(t)e_{k}\right\rangle \right|.
\]
First 

\[
\left|\left\langle U^{\ast}(t)\varphi_{n},c_{\hbar}^{\lambda}(x,\xi)U^{\ast}(t)(\varphi_{n}-e_{k})\right\rangle \right|^{2}
\]

\[
\leq\left\langle U^{\ast}(t)\varphi_{n},c_{\hbar}^{\lambda}(x,\xi)U^{\ast}(t)\varphi_{n}\right\rangle \left\langle U^{\ast}(t)(\varphi_{n}-e_{k}),c_{\hbar}^{\lambda}(x,\xi)U^{\ast}(t)(\varphi_{n}-e_{k})\right\rangle 
\]

\[
\leq\left\langle U^{\ast}(t)\varphi_{n},(\lambda^{2}\left|x\right|^{2}+\left|\xi\right|^{2}+\lambda^{2}\left|y\right|^{2}-\hbar^{2}\Delta_{y})U^{\ast}(t)\varphi_{n}\right\rangle \times
\]

\begin{equation}
\left\langle U^{\ast}(t)(\varphi_{n}-e_{k}),(\lambda^{2}\left|x\right|^{2}+\left|\xi\right|^{2}+\lambda^{2}\left|y\right|^{2}-\hbar^{2}\Delta_{y})U^{\ast}(t)(\varphi_{n}-e_{k})\right\rangle .\label{eq:-3}
\end{equation}
By Corollary \ref{domain of magnetic vs non magnetic } in Section
\ref{domainS SECTION} we have 

\[
\left|\left\langle U^{\ast}(t)(\varphi_{n}-e_{k}),(\lambda^{2}\left|x\right|^{2}+\left|\xi\right|^{2}+\lambda^{2}\left|y\right|^{2}-\hbar^{2}\Delta_{y})U^{\ast}(t)(\varphi_{n}-e_{k})\right\rangle \right|
\]

\[
\leq(1+\lambda^{2}\left|x\right|^{2}+\left|\xi\right|^{2})\left\Vert \varphi_{n}-e_{k}\right\Vert _{2}^{2}+C(\lambda)\left\Vert \mathscr{H}_{0}U^{\ast}(t)(\varphi_{n}-e_{k})\right\Vert _{2}^{2}
\]

\[
=(1+\lambda^{2}\left|x\right|^{2}+\left|\xi\right|^{2})\left\Vert \varphi_{n}-e_{k}\right\Vert _{2}^{2}+C(\lambda)\left\Vert (\mathscr{H}-V)U^{\ast}(t)(\varphi_{n}-e_{k})|\right\Vert _{2}^{2}
\]

\[
\leq(1+\lambda^{2}\left|x\right|^{2}+\left|\xi\right|^{2}+2\left\Vert V\right\Vert _{\infty}^{2}+2C(\lambda))\left\Vert \varphi_{n}-e_{k}\right\Vert _{2}^{2}+\left\Vert \mathscr{H}(\varphi_{n}-e_{k})|\right\Vert _{2}^{2}.
\]
Similarly
\[
\left\langle U^{\ast}(t)\varphi_{n},(\lambda^{2}\left|x\right|^{2}+\left|\xi\right|^{2}+\lambda^{2}\left|y\right|^{2}-\hbar^{2}\Delta_{y})U^{\ast}(t)\varphi_{n}\right\rangle 
\]

\[
\leq(1+\lambda^{2}\left|x\right|^{2}+\left|\xi\right|^{2})\left\Vert \varphi_{n}\right\Vert _{2}^{2}+\left\Vert \mathscr{H}_{0}\varphi_{n}\right\Vert _{2}^{2}\leq(1+\lambda^{2}\left|x\right|^{2}+\left|\xi\right|^{2}+2\left\Vert V\right\Vert _{\infty}^{2}+2C(\lambda))\left\Vert \varphi_{n}\right\Vert _{2}^{2}+\left\Vert \mathscr{H}\varphi_{n}\right\Vert _{2}^{2}
\]
which shows that the first factor in the right hand of (\ref{eq:-3})
is bounded (uniformly in $n$) and so the right hand side is 
\[
\lesssim(1+\lambda^{2}\left|x\right|^{2}+\left|\xi\right|^{2}+2\left\Vert V\right\Vert _{\infty}^{2}+2C(\lambda))\left\Vert \varphi_{n}-e_{k}\right\Vert _{2}^{2}+\left\Vert \mathscr{H}(\varphi_{n}-e_{k})|\right\Vert _{2}^{2}\underset{n\rightarrow\infty}{\rightarrow}0.
\]
\end{proof}

\section{\label{sec:Monge-Kantorovich-Convergence SECTION} Monge-Kantorovich
Convergence}

As will become clear in Section \ref{OBSERVATION SECTION }, it will
be of use to consider a cost function with a weighted classical part.
Namely, we insert a parameter $\lambda>0$ in the cost function in
order to optimize some constants which will show up in Section \ref{OBSERVATION SECTION }.
Define 

\[
c_{\hbar}^{\lambda}(x,\xi)\coloneqq\frac{1}{2}\lambda^{2}\left|x-y\right|^{2}+\frac{1}{2}\left|\xi+i\hbar\nabla_{y}\right|^{2}.
\]
Accordingly, define 

\[
E_{\hbar}^{\lambda}(\rho,R)\coloneqq\underset{Q\in\mathcal{C}(\rho,R)\cap\mathcal{D}^{2}(\mathfrak{H})}{\inf}\left(\underset{\mathbb{R}^{d}\times\mathbb{R}^{d}}{\int}\mathrm{trace}(\sqrt{Q(x,\xi)}c_{\hbar}^{\lambda}(x,\xi)\sqrt{Q(x,\xi)})dxd\xi\right)^{\frac{1}{2}}.
\]
In the following lemma we gather lower and upper bounds on $E_{\hbar}^{\lambda}(\rho,R_{\hbar})$
in terms of the Monge-Kantorovich distance 
\begin{lem}
\label{MK optimization } \textup{(Proposition 2.3 in {[}15{]} and
Section 3 in {[}14{]})} Let $f$ be a probability density on $\mathbb{R}^{d}\times\mathbb{R}^{d}$
and let $\mu$ be a Borel probability measure on $\mathbb{R}^{d}\times\mathbb{R}^{d}$
with 
\[
\underset{\mathbb{R}^{d}\times\mathbb{R}^{d}}{\int}\left(\left|x\right|^{2}+\left|\xi\right|^{2}\right)\mu(x,\xi)dxd\xi<\infty.
\]
\begin{flushleft}
1. If $R_{\hbar}=\mathrm{OP}_{\hbar}^{T}((2\pi\hbar)^{d}\mu)$ then
$R_{\hbar}\in\mathcal{D}^{2}(\mathfrak{H})$ and one has the following
bounds 
\par\end{flushleft}
\begin{flushleft}
\begin{equation}
E_{\hbar}^{\lambda}(f,R_{\hbar})^{2}\leq\max(1,\lambda^{2})\mathrm{dist}_{\mathrm{MK},2}(f,\mu)^{2}+\frac{1}{4}(\lambda^{2}+1)d\hbar.\label{eq:-12-1-1}
\end{equation}
\par\end{flushleft}
2. If $R\in\mathcal{D}^{2}(\mathfrak{H})$ then 

\[
E_{\hbar}^{\lambda}(f,R)^{2}\geq\mathrm{dist}_{\mathrm{MK},2}(f,\widetilde{W}_{\hbar}[R])^{2}-\frac{1}{4}(\lambda^{2}+1)d\hbar.
\]
\end{lem}
\textit{Proof of Theorem \ref{Main theorem 2}. }By Theorem \ref{Pseudo distance estimates 2},
Lemma \ref{MK optimization } and Lemma \ref{finite second quantum moments }
we have 

\[
\mathrm{dist}_{\mathrm{MK},2}(f(t),\widetilde{W}_{\hbar}[R(t)])^{2}-\frac{1}{4}(1+\lambda^{2})d\hbar\leq E_{\hbar}^{\lambda}(f(t),R(t))^{2}
\]

\[
\leq\beta(K,\lambda)e^{\alpha(L,K',\lambda,\rho_{0})t}E_{\hbar}(f^{\mathrm{in}},R_{\hbar}^{\mathrm{in}})^{2}
\]

\[
\leq\beta(K,\lambda)e^{\alpha(L,K',\lambda,\rho_{0})t}\left(\max(1,\lambda^{2})\mathrm{dist}_{\mathrm{MK},2}(f^{\mathrm{in}},\mu)^{2}+\frac{1}{4}(\lambda^{2}+1)d\hbar\right),
\]
which for $\lambda=1$ gives 
\[
\mathrm{dist}_{\mathrm{MK},2}(f(t),\widetilde{W}_{\hbar}[R(t)])^{2}-\frac{d\hbar}{2}\leq\beta(K)e^{\alpha(L,K',\rho_{0})t}\left(\mathrm{dist}_{\mathrm{MK},2}(f^{\mathrm{in}},\mu^{\mathrm{in}})^{2}+\frac{d\hbar}{2}\right).
\]

\begin{flushright}
$\square$
\par\end{flushright}

\begin{onehalfspace}

\end{onehalfspace}\textit{Proof of Theorem \ref{Main thm }. }By Theorem \ref{Pseudo distance estimates },
Lemma \ref{MK optimization } and Lemma \ref{finite second quantum moments }
we have 
\[
\mathrm{dist}_{\mathrm{MK},2}(f_{\epsilon}(t),\widetilde{W}_{\hbar}[R(t)])^{2}-\frac{1}{2}(1+\epsilon^{2})\hbar\leq E_{\hbar,\epsilon}(f_{\epsilon}(t),R_{\epsilon,\hbar}(t))^{2}
\]

\[
\leq e^{\max(2,\epsilon^{2}(1+L^{2}))t}E_{\hbar,\epsilon}(f_{\epsilon}^{\mathrm{in}},R_{\epsilon,\hbar}^{\mathrm{in}})^{2}
\]

\begin{flushright}
\[
\leq e^{\max(2,\epsilon^{2}(1+L^{2}))t}\left(\mathrm{dist}_{\mathrm{MK},2}(f_{\epsilon}^{\mathrm{in}},\mu)^{2}+\left(\frac{1+\epsilon^{2}}{2}\right)\hbar\right).
\]
\par\end{flushright}

\begin{flushright}
$\square$
\par\end{flushright}

\section{\label{OBSERVATION SECTION } Observation Inequality }

As an application of the method demonstrated in Section \ref{sec:Monge-Kantorovich-Convergence SECTION},
we can investigate the problem of observation type inequalities for
the von Neumann equation. Following the tradition, let us introduce
the notion of an observation inequality for the linear Schr\textcyr{\"\cyro}dinger
equation, whose Cauchy problem is 

\begin{equation}
i\hbar\partial_{t}\psi=-\frac{1}{2}\hbar^{2}\Delta\psi+V(y)\psi,\ \psi(0,x)=\psi^{\mathrm{in}}.\label{eq:schrodinger-1}
\end{equation}
As before we assume that $V\in C^{1,1}(\mathbb{R}^{d})$. An observation
inequality for equation (\ref{eq:schrodinger-1}) is an inequality
of the form 

\[
\left\Vert \psi^{in}\right\Vert _{2}^{2}\leq C_{\mathrm{OBS}}\stackrel[0]{T}{\int}\underset{\Omega}{\int}|\psi(t,x)|^{2}dxdt,
\]
for some $T>0,\Omega\subset\mathbb{R}^{d}$ an open set, $C_{\mathrm{OBS}}>0$
some constant and all initial data $\psi^{\mathrm{in}}$ which satisfies
some constraints related to $\Omega$. Note that the conservation
of the $L^{2}$ norm forces $C_{\mathrm{OBS}}\geq\frac{1}{T}$ (for
$\left\Vert \psi^{in}\right\Vert \neq0$). Observability inequalities
were first introduced in {[}21{]}, as a dual notion of controllability.
A more modern exposition on the subject can be found in {[}21{]}
in the context of the linear and nonlinear Schr\textcyr{\"\cyro}dinger
equation. We now introduce a geometric condition due to Bardos-Lebeau-Rauch
{[}4{]}, which is known to imply observability for equation (\ref{eq:schrodinger-1}).
Recall that the assumption $V\in C^{1,1}(\mathbb{R}^{d})$ implies
that Newton's system of ODEs (\ref{Newton magnetic}) has a unique
flow $(X(t),\Xi(t))$, and we denote $\Phi_{t}(x,\xi)=(X(t,x,\xi),\Xi(t,x,\xi))$. 
\begin{defn}
\label{BLR condition} Let $\mathcal{K}\subset\mathbb{R}^{d}\times\mathbb{R}^{d}$
be compact, $\Omega\subset\mathbb{R}^{d}$ open and $T>0$. The triplet
$(\mathcal{K},\Omega,T)$ is said to satisfy the \textit{Bardos-Lebeau-Rauch
geometric condition} (henceforth (GC)) if for each $(x,\xi)\in\mathcal{K}$
there is some $t=t(x,\xi)\in(0,T)$ such that $X(t,x,\xi)\in\Omega$. 
\end{defn}
Denote by $\Omega_{\delta}$ the $\delta$-neighborhood of $\Omega$,
i.e. 

\[
\Omega_{\delta}\coloneqq\left\{ x\in\mathbb{R}^{d}\left|\mathrm{dist}(x,\Omega)<\delta\right.\right\} .
\]
Adapting the methods of {[}15{]} (especially Theorem 4.1), we formulate
and prove a quantitative observation inequality for the von Neumann
equation (\ref{Hartree}), by utilizing the propagation estimate of
Section \ref{sec:5 OF CHAP 3}, as well as the optimal transport theory
of Section \ref{sec:Monge-Kantorovich-Convergence SECTION}. The proof
of the forthcoming theorem is almost identical to the proof in {[}15{]},
and is included for the purpose of clarifying the link between Theorem
(\ref{Main theorem 2}) and observability, which is not obvious--
indeed the optimal transport approach demonstrated in {[}15{]} (which
is the same approach used here) is not a standard tool in some of
the earlier literature on the subject.
\begin{thm}
\label{Observation inequality } Let $V$ satisfy (\ref{CONDITION ON V}).
Let $\mathcal{K}\subset\mathbb{R}^{d}\times\mathbb{R}^{d}$ be compact,
$\Omega\subset\mathbb{R}^{d}$ open and $T>0$. Suppose that the triplet
$(\mathcal{K},\Omega,T)$ satisfies (GC). Let $\chi\in\mathrm{Lip}(\mathbb{R}^{d})$
with $\chi(x)>0$ for all $x\in\Omega$. For each $t\geq0$ set

\[
R(t)=U^{\ast}(t)R^{\mathrm{in}}U(t),\ f(t,X,\Xi)=f^{\mathrm{in}}(\Phi_{-t}(X,\Xi)),
\]
where $R^{\mathrm{in}}=\mathrm{OP}_{\hbar}^{T}((2\pi\hbar)^{d}f^{\mathrm{in}})$
and $f^{\mathrm{in}}\in\mathcal{P}(\mathbb{R}^{d}\times\mathbb{R}^{d})$
is supported on $\mathcal{K}$ with $\underset{\mathbb{R}^{d}\times\mathbb{R}^{d}}{\int}\left(\left|x\right|^{2}+\left|\xi\right|^{2}\right)f^{\mathrm{in}}(x,\xi)dxd\xi<\infty$.
Then for all $\lambda>0$

\[
\stackrel[0]{T}{\int}\mathrm{trace}\left(\chi R(t)\right)dt
\]

\[
\geq\underset{(x,\xi)\in\mathcal{K}}{\inf}\stackrel[0]{T}{\int}\chi(X(t,x,\xi))dt-\frac{\mathrm{Lip}(\chi)}{\lambda}\sqrt{\beta(K,\lambda)}\sqrt{(\lambda^{2}+1)d\hbar}\frac{2\left(e^{\frac{\alpha(L,K',\lambda,\rho_{0},d)T}{2}}-1\right)}{\alpha(L,K',\lambda,\rho_{0},d)}.
\]
Consequently, 
\[
\stackrel[0]{T}{\int}\mathrm{trace}\left(\mathbf{1}_{\Omega_{\delta}}R(t)\right)dt\geq C_{\mathrm{OBS}}=C_{\mathrm{OBS}}\mathrm{trace}(R^{\mathrm{in}})
\]
for all $\delta>c$, where
\[
C_{\mathrm{OBS}}=C_{\mathrm{OBS}}(\mathcal{K},\Omega,T,\delta,\hbar,K,K',L,\rho_{0})>0
\]
and 

\[
c=c(\mathcal{K},\Omega,T,d,\hbar,K,K',L,\rho_{0})>0
\]
admit explicit formulas. 
\end{thm}
\textit{Proof. }\textbf{Step 1.} We compute that 

\[
\mathrm{trace}\left(\chi R(t)\right)-\underset{\mathbb{R}^{d}\times\mathbb{R}^{d}}{\int\int}\chi(x)f(t,x,\xi)dxd\xi
\]

\[
=\mathrm{trace}\left(\chi\underset{\mathbb{R}^{d}\times\mathbb{R}^{d}}{\int\int}Q(t,x,\xi)dxd\xi\right)-\underset{\mathbb{R}^{d}\times\mathbb{R}^{d}}{\int\int}\chi(x)\mathrm{trace}\left(Q(t,x,\xi)\right)dxd\xi
\]

\[
=\underset{\mathbb{R}^{d}\times\mathbb{R}^{d}}{\int\int}\mathrm{trace}\left((\chi-\chi(x))Q(t,x,\xi)\right)dxd\xi.
\]
Therefore 

\[
\left|\mathrm{trace}\left(\chi R(t)\right)-\underset{\mathbb{R}^{d}\times\mathbb{R}^{d}}{\int\int}\chi(x)f(t,x,\xi)dxd\xi\right|
\]

\[
\leq\underset{\mathbb{R}^{d}\times\mathbb{R}^{d}}{\int\int}\mathrm{trace}\left(\sqrt{Q(t,x,\xi)}\left|\chi-\chi(x)\right|\sqrt{Q(t,x,\xi)}\right)dxd\xi
\]

\[
\leq\mathrm{Lip}(\chi)\underset{\mathbb{R}^{d}\times\mathbb{R}^{d}}{\int\int}\mathrm{trace}\left(\sqrt{Q(t,x,\xi)}\left|x-y\right|\sqrt{Q(t,x,\xi)}\right)dxd\xi
\]

\[
\leq\mathrm{Lip}(\chi)\underset{\mathbb{R}^{d}\times\mathbb{R}^{d}}{\int\int}\mathrm{trace}\left(\sqrt{Q(t,x,\xi)}\left(\eta\left|x-y\right|^{2}+\frac{1}{\eta}\right)\sqrt{Q(t,x,\xi)}\right)dxd\xi
\]

\[
=\mathrm{Lip}(\chi)\left(\eta\underset{\mathbb{R}^{d}\times\mathbb{R}^{d}}{\int\int}\mathrm{trace}\left(\sqrt{Q(t,x,\xi)}\left|x-y\right|^{2}\sqrt{Q(t,x,\xi)}\right)dxd\xi+\frac{1}{\eta}\right).
\]
Now we minimize the right hand side with respect to $\eta$ by taking 

\[
\eta\coloneqq\left(\underset{\mathbb{R}^{d}\times\mathbb{R}^{d}}{\int\int}\mathrm{trace}\left(\sqrt{Q(t,x,\xi})\left|x-y\right|^{2}\sqrt{Q(t,x,\xi)}\right)dxd\xi\right)^{-\frac{1}{2}},
\]
which produces 

\[
\left|\mathrm{trace}\left(\chi R(t)\right)-\underset{\mathbb{R}^{d}\times\mathbb{R}^{d}}{\int\int}\chi(x)f(t,x,\xi)dxd\xi\right|
\]

\[
\leq2\mathrm{Lip}(\chi)\left(\underset{\mathbb{R}^{d}\times\mathbb{R}^{d}}{\int\int}\mathrm{trace}\left(\sqrt{Q(t,x,\xi)}\left|x-y\right|^{2}\sqrt{Q(t,x,\xi)}\right)dxd\xi\right)^{\frac{1}{2}}
\]

\[
\leq\frac{2\sqrt{2}\mathrm{Lip}(\chi)}{\lambda}\left(\underset{\mathbb{R}^{d}\times\mathbb{R}^{d}}{\int\int}\mathrm{trace}\left(\sqrt{Q(t,x,\xi)}\left(\frac{\lambda^{2}}{2}\left|x-y\right|^{2}+\frac{1}{2}\left|\xi+i\hbar\nabla_{y}\right|^{2}\right)\sqrt{Q(t,x,\xi)}\right)dxd\xi\right)^{\frac{1}{2}}
\]

\begin{equation}
=\frac{2\sqrt{2}\mathrm{Lip}(\chi)}{\lambda}\sqrt{\mathcal{E}_{\hbar}^{\lambda}(t)}\leq\frac{2\sqrt{2}\mathrm{Lip}(\chi)}{\lambda}\sqrt{\beta(K,\lambda)}e^{\frac{\alpha(L,K',\lambda,\rho_{0},d)t}{2}}\sqrt{\mathcal{E}_{\hbar}^{\lambda}(0)},\label{eq:-6}
\end{equation}
where the last inequality is due to step 3 in Section \ref{sec:5 OF CHAP 3}.
Minimizing inequality (\ref{eq:-6}) on all $Q^{\mathrm{in}}\in\mathcal{C}(f^{\mathrm{in}},R^{\mathrm{in}})$
yields 

\[
\left|\mathrm{trace}(\chi R(t))-\underset{\mathbb{R}^{d}\times\mathbb{R}^{d}}{\int\int}\chi(x)f(t,x,\xi)dxd\xi\right|
\]

\begin{equation}
\leq\frac{2\sqrt{2}\mathrm{Lip}(\chi)}{\lambda}\sqrt{\beta(K,\lambda)}e^{\frac{\alpha(L,K',\lambda,\rho_{0})t}{2}}E_{\hbar}^{\lambda}(f^{\mathrm{in}},R^{\mathrm{in}}).\label{eq:-3-3}
\end{equation}

\textbf{Step 2.} We have 

\[
\underset{\mathbb{R}^{d}\times\mathbb{R}^{d}}{\int\int}\chi(x)f(t,x,\xi)dxd\xi=\underset{\mathbb{R}^{d}\times\mathbb{R}^{d}}{\int\int}\chi(x)f^{\mathrm{in}}(\Phi_{-t}(x,\xi))dxd\xi
\]

\[
=\underset{\mathbb{R}^{d}\times\mathbb{R}^{d}}{\int\int}\chi(X(t,x,\xi))f^{\mathrm{in}}(x,\xi)dxd\xi.
\]

Therefore inequality (\ref{eq:-3-3}) implies
\[
\stackrel[0]{T}{\int}\mathrm{trace}(\chi R(t))dt
\]

\[
\geq\underset{\mathbb{R}^{d}\times\mathbb{R}^{d}}{\int\int}\left(\stackrel[0]{T}{\int}\chi(X(t,x,\xi))dt\right)f^{in}(x,\xi)dxd\xi-\frac{2\sqrt{2}\mathrm{Lip}(\chi)}{\lambda}\sqrt{\beta(K,\lambda)}E_{\hbar}^{\lambda}(f^{in},R^{in})\stackrel[0]{T}{\int}e^{\frac{\alpha(L,K',\lambda,\rho_{0})t}{2}}dt
\]

\[
=\underset{\mathcal{K}}{\int\int}\left(\stackrel[0]{T}{\int}\chi(X(t,x,\xi))dt\right)f^{in}(x,\xi)dxd\xi-\frac{2\sqrt{2}\mathrm{Lip}(\chi)}{\lambda}\sqrt{\beta(K,\lambda)}E_{\hbar}^{\lambda}(f^{in},R^{in})\frac{2(e^{\frac{\alpha(L,K',\lambda,\rho_{0})T}{2}}-1)}{\alpha(L,K',\lambda,\rho_{0})}
\]

\[
\geq\underset{(x,\xi)\in\mathcal{K}}{\inf}\stackrel[0]{T}{\int}\chi(X(t,x,\xi))dt-\frac{2\sqrt{2}\mathrm{Lip}(\chi)}{\lambda}\sqrt{\beta(K,\lambda)}E_{\hbar}^{\lambda}(f^{in},R^{in})\frac{2(e^{\frac{\alpha(L,K',\lambda,\rho_{0})T}{2}}-1)}{\alpha(L,K',\lambda,\rho_{0})}
\]

\begin{equation}
\geq\underset{(x,\xi)\in\mathcal{K}}{\inf}\stackrel[0]{T}{\int}\chi(X(t,x,\xi))dt-\frac{\sqrt{2}\mathrm{Lip}(\chi)}{\lambda}\sqrt{\beta(K,\lambda)}\sqrt{(\lambda^{2}+1)d\hbar}\frac{2(e^{\frac{\alpha(L,K',\lambda,\rho_{0})T}{2}}-1)}{\alpha(L,K',\lambda,\rho_{0})},\label{eq:-9-2}
\end{equation}
where the last inequality is due to Lemma \ref{MK optimization }. 

\textbf{Step 3. }We explain how to remove the cutoff $\chi$ in step
2, which is the second statement of Theorem \ref{Observation inequality }.
Since $\Omega$ is open the indicator $\mathbf{1}_{\Omega}$ is lower
semicontinuous. By condition (GC), for each $(x,\xi)\in\mathcal{K}$
there is some $t(x,\xi)\in(0,T)$ such that $\mathbf{1}_{\Omega}(X(t(x,\xi),x,\xi))=1$.
For each $(x,\xi)\in\mathcal{K}$ consider the set 

\[
S(x,\xi)\coloneqq\left\{ t\in(0,T)\left|\mathbf{1}_{\Omega}(X(t,x,\xi))>\frac{1}{2}\right.\right\} .
\]
Evidently $S(x,\xi)$ is open. Therefore there is some $\tau(x,\xi)$
such that $[t(x,\xi)-\tau(x,\xi),t(x,\xi)+\tau(x,\xi)]\subset S(x,\xi)$
which means 

\[
\stackrel[0]{T}{\int}\mathbf{1}_{\Omega}(X(t(x,\xi),x,\xi))dt
\]

\begin{equation}
\geq\underset{[t(x,\xi)-\tau(x,\xi),t(x,\xi)+\tau(x,\xi)]}{\int}\mathbf{1}_{\Omega}(X(t(x,\xi),x,\xi))dt\geq\tau(x,\xi)>0.\label{eq:-2-3}
\end{equation}
Furthermore, by Fatou's lemma the function 

\[
(x,\xi)\mapsto\stackrel[0]{T}{\int}\mathbf{1}_{\Omega}(X(t(x,\xi),x,\xi))dt
\]
is lower semicontinuous on $\mathcal{K}$, and positive on $\mathcal{K}$
because of (\ref{eq:-2-3}). As $\mathcal{K}$ is compact, set 

\[
C(T,\mathcal{K},\Omega)\coloneqq\underset{(x,\xi)\in\mathcal{K}}{\inf}\stackrel[0]{T}{\int}\mathbf{1}_{\Omega}(X(t(x,\xi),x,\xi))dt>0,
\]
and put 

\[
\chi_{\delta}(x)=\left(1-\frac{\mathrm{dist}(x,\Omega)}{\delta}\right)_{+}.
\]
Note that $\chi_{\delta}\in\mathrm{Lip}(\mathbb{R}^{d})$ and $\mathrm{Lip}(\chi_{\delta})=\frac{1}{\delta}$.
Clearly $\mathbf{1}_{\Omega}\leq\chi_{\delta}\leq\mathbf{1}_{\Omega_{\delta}}$
and consequently 

\[
\stackrel[0]{T}{\int}\mathrm{trace}\left(\chi_{\delta}R(t)\right)dt\leq\stackrel[0]{T}{\int}\mathrm{trace}\left(\mathbf{1}_{\Omega_{\delta}}R(t)\right)dt.
\]
Thus, in view of inequality (\ref{eq:-9-2}) applied for $\chi_{\delta}$
we get

\[
\stackrel[0]{T}{\int}\mathrm{trace}\left(\mathbf{1}_{\Omega_{\delta}}R(t)\right)dt\geq C(T,\mathcal{K},\Omega)-\frac{\sqrt{2}}{\lambda\delta}\sqrt{\beta(K,\lambda)}\sqrt{(\lambda^{2}+1)d\hbar}\frac{2\left(e^{\frac{\alpha(L,K',\lambda,\rho_{0})T}{2}}-1\right)}{\alpha(L,K',\lambda,\rho_{0})}.
\]
In order to maximize the right hand side of the last inequality put

\[
C^{\ast}(T,K,K',L,d,\hbar)\coloneqq\underset{\lambda>0}{\inf}\frac{1}{\lambda}\sqrt{\beta(K,\lambda)}\sqrt{(\lambda^{2}+1)d\hbar}\frac{2\left(e^{\frac{\alpha(L,K',\lambda,\rho_{0})T}{2}}-1\right)}{\alpha(L,K',\lambda,\rho_{0})}.
\]
If we pick $\delta>0$ so that 

\[
C_{\mathrm{OBS}}\coloneqq C(T,\mathcal{K},\Omega)-\frac{1}{\delta}C^{\ast}(T,K,K',L,d,\hbar)>0,
\]
 then we finally get 

\[
\stackrel[0]{T}{\int}\mathrm{trace}\left(\mathbf{1}_{\Omega_{\delta}}R(t)\right)dt\geq C_{\mathrm{OBS}}=C_{\mathrm{OBS}}\mathrm{trace}(R^{\mathrm{in}}).
\]

\begin{onehalfspace}
\begin{flushright}
$\square$
\par\end{flushright}
\end{onehalfspace}

\section{\label{domainS SECTION} On the Domain of the Magnetic/Non-Magnetic
Harmonic Oscillator}

Roughly speaking, Section \ref{sec:5 OF CHAP 3} and Section \ref{sec:3 OF CHAP 3}
focused on how the presence of a magnetic field influences the formal
calculations leading to the evolution inequality for $E_{\hbar}$.
Interestingly, the presence of a magnetic field also influences the
spectral theory which is required in order to put these formal calculations
on rigorous grounds. In the proof of both of the main Theorems \ref{Main theorem 2}
and \ref{Main thm } we relied on the fact that the domain of the
magnetic harmonic oscillator identifies with the domain of the (non-magnetic)
harmonic oscillator. This is quite apparent in the case where $A$
is bounded, but more subtle for sublinear $A$, which is the case
of interest. The purpose of this section is to review the relevant
literature, as well as state some bounds which are presumably not
new, nevertheless do not appear explicitly enough in the literature.
We start by elaborating on the essential self-adjointness of the cost
function. Recall the notations 

\[
\mathscr{H}_{0}\coloneqq\frac{1}{2}\left|-i\hbar\nabla_{x}+A(y)\right|^{2}+\frac{1}{2}\left|y\right|^{2}
\]

\[
\mathbf{K}\coloneqq\frac{1}{2}\left|-i\hbar\nabla_{y}+A(y)\right|^{2}
\]
and 
\[
\Pi_{k}\coloneqq-i\hbar\partial_{k}+A_{k}(y).
\]
The following theorem is a restatement of Theorem 1.1 in {[}27{]}
for the specific settings of interest
\begin{thm}
Let $A\in\mathrm{Lip}_{\mathrm{loc}}(\mathbb{R}^{d})$ and $W\in L_{\mathrm{loc}}^{2}(\mathbb{R}^{d})$
with $W\geq0$. Then $\frac{1}{2}\left|-i\hbar\nabla_{y}+A(y)\right|^{2}+\frac{1}{2}W(y)$
is essentially self-adjoint on $C_{0}^{\infty}(\mathbb{R}^{d})$.
The domain of $\frac{1}{2}\left|-i\hbar\nabla_{y}+A(y)\right|^{2}+\frac{1}{2}W(y)$
is given by 

\[
D=\left\{ \varphi\in L^{2}(\mathbb{R}^{d})\left|\frac{1}{2}\left|-i\hbar\nabla_{y}+A(y)\right|^{2}+\frac{1}{2}W(y)\in L^{2}(\mathbb{R}^{d})\}\right.\right\} .
\]
\end{thm}
An immediate consequence is that both operators

\[
\widetilde{c}_{\hbar}^{\lambda}(x,\xi)=\frac{\lambda^{2}}{2}\left|x-y\right|^{2}+\frac{1}{2}\left|\xi+A(x)-A(y)+i\hbar\nabla_{y}\right|^{2}
\]

\[
c_{\hbar}^{\lambda}(x,\xi)=\frac{\lambda^{2}}{2}\left|x-y\right|^{2}+\frac{1}{2}\left|\xi+i\hbar\nabla_{y}\right|^{2}
\]
are essentially self-adjoint on $C_{0}^{\infty}(\mathbb{R}^{d})$
with domains 

\[
D\left(\widetilde{c}_{\hbar}^{\lambda}(x,\xi)\right)=\left\{ \varphi\in\mathfrak{H}\left|\widetilde{c}_{\hbar}^{\lambda}(x,\xi)\varphi\in\mathfrak{H}\right.\right\} 
\]
and 

\[
D\left(c_{\hbar}^{\lambda}(x,\xi)\right)=\left\{ \varphi\in\mathfrak{H}\left|c_{\hbar}^{\lambda}(x,\xi)\varphi\in\mathfrak{H}\right.\right\} 
\]
respectively. In addition, it is an exercise to check (see e.g. Footnote
3 in {[}7{]}) that the domain of the harmonic oscillator $\mathscr{O}\coloneqq-\frac{\hbar^{2}}{2}\Delta+\frac{1}{2}\left|y\right|^{2}$
is characterized as 

\[
D(\mathscr{O})=H^{2}(\mathbb{R}^{d})\cap\left\{ \varphi\in\mathfrak{H}\left|\left|y\right|^{2}\varphi\in\mathfrak{H}\right.\right\} .
\]
The characterization of $D\left(\mathscr{H}_{0}\right)$ is a somewhat
more challenging task. We proceed by explaining how to compare $D\left(\mathscr{H}_{0}\right)$
and $D(\mathscr{O})$-- eventually we wish to show they are the same.
Let us confine ourselves to the case $d\geq3$ (see Remark \ref{remark 7.7}
for the case $d=2$).We start with 
\begin{lem}
\label{estimate on moments and gradient } Let $A$ satisfy $(\mathbb{A}')$
with constant $K$. For each $\varphi\in C_{0}^{\infty}(\mathbb{R}^{d})$
one has the following estimates 
\[
\underset{\mathbb{R}^{d}}{\int}\left|x\right|^{2}\left|\varphi\right|^{2}(x)dx\leq\left\Vert \mathscr{H}_{0}\varphi\right\Vert _{2}^{2}+\left\Vert \varphi\right\Vert _{2}^{2}
\]
 and 

\[
\underset{\mathbb{R}^{d}}{\int}\hbar^{2}\left|\nabla\varphi\right|^{2}(x)dx\leq\max(4,K^{2})\left(\left\Vert \mathscr{H}_{0}\varphi\right\Vert _{2}^{2}+\left\Vert \varphi\right\Vert _{2}^{2}\right).
\]
\end{lem}
\textit{Proof. }One has 

\[
\underset{\mathbb{R}^{d}}{\int}\left|x\right|^{2}\left|\varphi\right|^{2}(x)dx\leq\left\langle \mathbf{K}\varphi,\varphi\right\rangle +\underset{\mathbb{R}^{d}}{\int}\left|x\right|^{2}\left|\varphi\right|^{2}(x)dx
\]

\[
=\underset{\mathbb{R}^{d}}{\int}\overline{\varphi}(x)\left(\mathbf{K}+\left|x\right|^{2}\right)\varphi(x)dx
\]

\begin{equation}
\leq2\left\Vert \mathscr{H}_{0}\varphi\right\Vert _{2}\left\Vert \varphi\right\Vert _{2}\leq\left\Vert \mathscr{H}_{0}\varphi\right\Vert _{2}^{2}+\left\Vert \varphi\right\Vert _{2}^{2},\label{eq:-29-1-1}
\end{equation}
which is the first inequality. The second inequality is implied from
inequality (\ref{eq:-29-1-1}) as follows 

\[
\underset{\mathbb{R}^{d}}{\int}\hbar^{2}\left|\partial_{k}\varphi\right|^{2}(x)dx=\left\Vert \left(\Pi_{k}-A_{k}(x)\right)\varphi\right\Vert _{2}^{2}
\]

\[
\leq2\left\Vert \Pi_{k}\varphi\right\Vert _{2}^{2}+K^{2}\underset{\mathbb{R}^{d}}{\int}\left|x\right|^{2}\left|\varphi\right|^{2}(x)dx
\]

\[
=4\underset{\mathbb{R}^{d}}{\int}\overline{\varphi}(x)\Pi_{k}^{2}\varphi(x)dx+K^{2}\underset{\mathbb{R}^{d}}{\int}\left|x\right|^{2}\left|\varphi\right|^{2}(x)dx
\]

\[
\leq\max(8,2K^{2})\underset{\mathbb{R}^{d}}{\int}\overline{\varphi}(x)\left(\frac{1}{2}\Pi_{k}^{2}+\frac{1}{2}\left|x\right|^{2}\right)\varphi(x)dx
\]

\[
\leq\max(4,K^{2})\left(\left\Vert \mathscr{H}_{0}\varphi\right\Vert _{2}^{2}+\left\Vert \varphi\right\Vert _{2}^{2}\right),
\]
which is the second inequality. 
\begin{flushright}
$\square$
\par\end{flushright}

The following ``magnetic maximal inequality'' is a specification
of Theorems 2.10 and 4.1 in {[}17{]} to the scenario considered in
the present work. 
\begin{lem}
\label{magnetic maximal inequality } Let $A$ satisfy $(\mathbb{A}')$.
There is a constant $C>0$ such that for all $\varphi\in C_{0}^{\infty}(\mathbb{R}^{d})$
it holds that 

\[
\left\Vert \mathbf{K}\varphi\right\Vert _{2}^{2}\leq C\left(\left\Vert \mathscr{H}_{0}\varphi\right\Vert _{2}^{2}+\left\Vert \varphi\right\Vert _{2}^{2}\right)
\]
and 

\[
\left\Vert \left|y\right|^{2}\varphi\right\Vert _{2}^{2}\leq C\left(\left\Vert \mathscr{H}_{0}\varphi\right\Vert _{2}^{2}+\left\Vert \varphi\right\Vert _{2}^{2}\right).
\]
\end{lem}
Next we obtain an estimate for a weighted $L^{2}$ norm of the gradient
in terms of the norm attached to $\mathscr{H}_{0}$. 
\begin{lem}
\label{weighted gradient estimate } Let $A$ satisfy $(\mathbb{A}')$.
There is a constant $C>0$ such that for all $\varphi\in C_{0}^{\infty}(\mathbb{R}^{d})$
it holds that 
\[
\underset{\mathbb{R}^{d}}{\int}\hbar^{2}\left|x\right|^{2}\left|\nabla\varphi\right|^{2}(x)dx\leq C\left(\left\Vert \mathscr{H}_{0}\varphi\right\Vert _{2}^{2}+\left\Vert \varphi\right\Vert _{2}^{2}\right).
\]
\end{lem}
\begin{proof}
We manipulate the integral 
\[
\underset{\mathbb{R}^{d}}{\int}\left|x\right|^{2}\overline{\varphi}(x)\mathscr{H}_{0}\varphi(x)dx
\]
as follows: 

\[
\underset{\mathbb{R}^{d}}{\int}\left|x\right|^{2}\overline{\varphi}(x)\mathscr{H}_{0}\varphi(x)dx=\frac{1}{2}\underset{\mathbb{R}^{d}}{\int}\left|x\right|^{2}\overline{\varphi}(x)\Pi_{k}^{2}\varphi(x)dx+\frac{1}{2}\underset{\mathbb{R}^{d}}{\int}\left|x\right|^{4}\left|\varphi\right|^{2}(x)dx
\]

\[
=\frac{1}{2}\underset{\mathbb{R}^{d}}{\int}\overline{\Pi_{k}\left(\left|x\right|^{2}\varphi\right)}(x)\Pi^{k}\varphi(x)dx+\frac{1}{2}\underset{\mathbb{R}^{d}}{\int}\left|x\right|^{4}\left|\varphi\right|^{2}(x)dx
\]

\[
=\frac{1}{2}\underset{\mathbb{R}^{d}}{\int}\overline{-\left|x\right|^{2}i\hbar\partial_{k}\varphi-2i\hbar x_{k}\varphi+A_{k}(x)\left|x\right|^{2}\varphi}\left(-i\hbar\partial_{x_{k}}+A_{k}(x)\right)\varphi(x)dx+\frac{1}{2}\underset{\mathbb{R}^{d}}{\int}\left|x\right|^{4}\left|\varphi\right|^{2}(x)dx
\]

\[
=\frac{1}{2}\underset{\mathbb{R}^{d}}{\int}\hbar^{2}\left|x\right|^{2}\left|\partial_{x_{k}}\varphi\right|^{2}(x)dx+\frac{1}{2}\underset{\mathbb{R}^{d}}{\int}\left|x\right|^{2}i\hbar\partial_{x_{k}}\overline{\varphi}A_{k}(x)\varphi(x)dx+\underset{\mathbb{R}^{d}}{\int}(i\hbar x_{k}\overline{\varphi})(-i\hbar\partial_{x_{k}}\varphi)(x)dx
\]

\[
+\underset{\mathbb{R}^{d}}{\int}i\hbar x_{k}\left|\varphi\right|^{2}(x)A_{k}(x)dx+\frac{1}{2}\underset{\mathbb{R}^{d}}{\int}\left|A_{k}\right|^{2}\left|x\right|^{2}\left|\varphi\right|^{2}(x)dx
\]

\[
-\frac{1}{2}\underset{\mathbb{R}^{d}}{\int}A_{k}(x)\left|x\right|^{2}\overline{\varphi}(x)i\hbar\partial_{x_{k}}\varphi(x)dx+\frac{1}{2}\underset{\mathbb{R}^{d}}{\int}\left|x\right|^{4}\left|\varphi\right|^{2}(x)dx
\]

\[
=\frac{1}{2}\underset{\mathbb{R}^{d}}{\int}\hbar^{2}\left|x\right|^{2}\left|\partial_{x_{k}}\varphi\right|^{2}(x)dx+\Re\left(\underset{\mathbb{R}^{d}}{\int}\left|x\right|^{2}i\hbar\partial_{x_{k}}\overline{\varphi}A_{k}(x)\varphi(x)dx\right)+\underset{\mathbb{R}^{d}}{\int}\hbar^{2}x_{k}\overline{\varphi}(x)\partial_{x_{k}}\varphi(x)dx
\]

\[
+\underset{\mathbb{R}^{d}}{\int}i\hbar x_{k}\left|\varphi\right|^{2}(x)A_{k}(x)dx+\frac{1}{2}\underset{\mathbb{R}^{d}}{\int}\left|A_{k}\right|^{2}\left|x\right|^{2}\left|\varphi\right|^{2}(x)dx+\frac{1}{2}\underset{\mathbb{R}^{d}}{\int}\left|x\right|^{4}\left|\varphi\right|^{2}(x)dx.
\]
We equate the real part of the left hand side with the real part of
the right hand side in order to find 

\[
\Re\left(\underset{\mathbb{R}^{d}}{\int}\left|x\right|^{2}\overline{\varphi}(x)\mathscr{H}_{0}\varphi(x)dx\right)
\]

\[
=\frac{1}{2}\underset{\mathbb{R}^{d}}{\int}\hbar^{2}\left|x\right|^{2}\left|\partial_{x_{k}}\varphi\right|^{2}(x)dx+\Re\left(\underset{\mathbb{R}^{d}}{\int}\left|x\right|^{2}i\hbar\partial_{x_{k}}\overline{\varphi}A_{k}(x)\varphi(x)dx\right)+\Re\left(\underset{\mathbb{R}^{d}}{\int}\hbar^{2}x_{k}\overline{\varphi}(x)\partial_{x_{k}}\varphi(x)dx\right)
\]

\begin{equation}
+\frac{1}{2}\underset{\mathbb{R}^{d}}{\int}\left|A_{k}\right|^{2}\left|x\right|^{2}\left|\varphi\right|^{2}(x)dx+\frac{1}{2}\underset{\mathbb{R}^{d}}{\int}\left|x\right|^{4}\left|\varphi\right|^{2}(x)dx.\label{eq:-2}
\end{equation}
By Young's inequality and the assumption on $A$ we can bound from
below the second and third terms in the right hand side of equation
(\ref{eq:-2}):

\[
\Re\left(\underset{\mathbb{R}^{d}}{\int}\left|x\right|^{2}i\hbar\partial_{x_{k}}\overline{\varphi}A_{k}(x)\varphi(x)dx\right)
\]

\[
\geq-\frac{1}{4}\underset{\mathbb{R}^{d}}{\int}\hbar^{2}\left|x\right|^{2}\left|\partial_{x_{k}}\varphi\right|^{2}(x)dx-\underset{\mathbb{R}^{d}}{\int}\left|x\right|^{2}\left|A_{k}(x)\right|^{2}\left|\varphi\right|^{2}(x)dx
\]

\begin{equation}
\geq-\frac{1}{4}\underset{\mathbb{R}^{d}}{\int}\hbar^{2}\left|x\right|^{2}\left|\partial_{x_{k}}\varphi\right|^{2}(x)dx-K^{2}\underset{\mathbb{R}^{d}}{\int}\left|x\right|^{4}\left|\varphi\right|^{2}(x)dx,\label{eq:}
\end{equation}
and 

\begin{equation}
\Re\left(\underset{\mathbb{R}^{d}}{\int}\hbar^{2}x_{k}\overline{\varphi}(x)\partial_{x_{k}}\varphi(x)dx\right)\geq-\frac{1}{2}\underset{\mathbb{R}^{d}}{\int}\hbar^{2}\left|x\right|^{2}\left|\varphi\right|^{2}(x)-\frac{1}{2}\underset{\mathbb{R}^{d}}{\int}\hbar^{2}\left|\partial_{x_{k}}\varphi\right|^{2}(x)dx.\label{eq:-1}
\end{equation}
Inequalities (\ref{eq:}) and (\ref{eq:-1}) together with identity
(\ref{eq:-2}) imply the following inequality 

\[
\frac{1}{2}\left\Vert \left|x\right|^{2}\varphi\right\Vert _{2}^{2}+\frac{1}{2}\left\Vert \mathscr{H}_{0}\varphi\right\Vert _{2}^{2}
\]

\[
\geq\frac{1}{4}\underset{\mathbb{R}^{d}}{\int}\hbar^{2}\left|x\right|^{2}\left|\partial_{x_{k}}\varphi\right|^{2}(x)dx-K^{2}\underset{\mathbb{R}^{d}}{\int}\left|x\right|^{4}\left|\varphi\right|^{2}(x)dx
\]

\[
-\frac{1}{2}\underset{\mathbb{R}^{d}}{\int}\hbar^{2}\left|x\right|^{2}\left|\varphi\right|^{2}(x)-\frac{1}{2}\underset{\mathbb{R}^{d}}{\int}\hbar^{2}\left|\partial_{x_{k}}\varphi\right|^{2}(x)dx
\]

\[
+\frac{1}{2}\underset{\mathbb{R}^{d}}{\int}\left|A_{k}\right|^{2}\left|x\right|^{2}\left|\varphi\right|^{2}(x)dx+\frac{1}{2}\underset{\mathbb{R}^{d}}{\int}\left|x\right|^{4}\left|\varphi\right|^{2}(x)dx
\]

\[
\geq\frac{1}{4}\underset{\mathbb{R}^{d}}{\int}\hbar^{2}\left|x\right|^{2}\left|\partial_{x_{k}}\varphi\right|^{2}(x)dx-K^{2}\underset{\mathbb{R}^{d}}{\int}\left|x\right|^{4}\left|\varphi\right|^{2}(x)dx
\]

\[
-\frac{1}{2}\underset{\mathbb{R}^{d}}{\int}\hbar^{2}\left|x\right|^{2}\left|\varphi\right|^{2}(x)dx-\frac{1}{2}\underset{\mathbb{R}^{d}}{\int}\hbar^{2}\left|\partial_{x_{k}}\varphi\right|^{2}(x)dx,
\]
which is recast as 

\[
\frac{1}{4}\underset{\mathbb{R}^{d}}{\int}\hbar^{2}\left|x\right|^{2}\left|\partial_{x_{k}}\varphi\right|^{2}(x)dx\leq\frac{1}{2}\left\Vert \mathscr{H}_{0}\varphi\right\Vert _{2}^{2}+\left(\frac{1}{2}+K^{2}\right)\underset{\mathbb{R}^{d}}{\int}\left|x\right|^{4}\left|\varphi\right|^{2}(x)dx
\]

\[
+\frac{1}{2}\underset{\mathbb{R}^{d}}{\int}\hbar^{2}\left|x\right|^{2}\left|\varphi\right|^{2}(x)dx+\frac{1}{2}\underset{\mathbb{R}^{d}}{\int}\hbar^{2}\left|\partial_{x_{k}}\varphi\right|^{2}(x)dx.
\]
Invoking Lemma \ref{magnetic maximal inequality } and Lemma \ref{estimate on moments and gradient }
we get 

\[
\underset{\mathbb{R}^{d}}{\int}\hbar^{2}\left|x\right|^{2}\left|\partial_{x_{k}}\varphi\right|^{2}(x)dx
\]

\[
\leq\frac{1}{2}\left\Vert \mathscr{H}_{0}\varphi\right\Vert _{2}^{2}+\left(\frac{1}{2}+K^{2}\right)C\left(\left\Vert \mathscr{H}_{0}\varphi\right\Vert _{2}^{2}+\left\Vert \varphi\right\Vert _{2}^{2}\right)
\]

\[
+\frac{\hbar^{2}}{2}\left(\left\Vert \mathscr{H}_{0}\varphi\right\Vert _{2}^{2}+\left\Vert \varphi\right\Vert _{2}^{2}\right)+\max\left(2,\frac{K^{2}}{2}\right)\left(\left\Vert \mathscr{H}_{0}\varphi\right\Vert _{2}^{2}+\left\Vert \varphi\right\Vert _{2}^{2}\right)
\]

\[
\leq C\left(\left\Vert \mathscr{H}_{0}\varphi\right\Vert _{2}^{2}+\left\Vert \varphi\right\Vert _{2}^{2}\right),
\]
as asserted. 
\end{proof}
As a corollary we find that the norm attached to the magnetic harmonic
oscillator is equivalent to the norm attached to the (non-magnetic)
harmonic oscillator 
\begin{cor}
\label{domain of magnetic vs non magnetic } There is a constant $C>0$
such that for all $\varphi\in C_{0}^{\infty}(\mathbb{R}^{d})$ it
holds that 
\[
\frac{1}{C}\left(\left\Vert \mathscr{H}_{0}\varphi\right\Vert _{2}^{2}+\left\Vert \varphi\right\Vert _{2}^{2}\right)\leq\left\Vert \mathscr{O}\varphi\right\Vert _{2}^{2}+\left\Vert \varphi\right\Vert _{2}^{2}\leq C\left(\left\Vert \mathscr{H}_{0}\varphi\right\Vert _{2}^{2}+\left\Vert \varphi\right\Vert _{2}^{2}\right).
\]
\end{cor}
\textit{Proof.} For each $\varphi\in C_{0}^{\infty}(\mathbb{R}^{d})$

\[
\left\Vert \Pi_{k}\varphi\right\Vert _{2}^{2}=\left\langle \Pi_{k}\varphi,\Pi_{k}\varphi\right\rangle 
\]

\[
=\left\langle \varphi,\Pi_{k}^{2}\varphi\right\rangle \leq\frac{\left\Vert \varphi\right\Vert _{2}^{2}+\left\Vert \Pi_{k}^{2}\varphi\right\Vert _{2}^{2}}{2}.
\]
Consequently (here we use again Einstein summation) 

\[
\left\Vert \mathscr{O}\varphi\right\Vert _{2}^{2}=\left\Vert -\frac{\hbar^{2}}{2}\Delta\varphi+\frac{1}{2}\left|y\right|^{2}\varphi\right\Vert _{2}^{2}
\]

\[
=\left\Vert \left(\frac{1}{2}(-i\hbar\partial_{y_{k}}+A^{k}(y)-A^{k}(y))^{2}+\frac{1}{2}\left|y\right|^{2}\right)\varphi\right\Vert _{2}^{2}
\]

\[
=\left\Vert \left(\frac{1}{2}\Pi_{k}^{2}-\frac{1}{2}\Pi_{k}\vee A_{k}(y)+\frac{1}{2}\left|A_{k}(y)\right|^{2}+\frac{1}{2}\left|y\right|^{2}\right)\varphi\right\Vert _{2}^{2}
\]

\[
\leq\left(\left\Vert \mathbf{K}\varphi\right\Vert _{2}+\frac{3K^{2}+1}{2}\left\Vert \left|y\right|^{2}\varphi\right\Vert _{2}+\frac{\hbar}{2}K\left\Vert \varphi\right\Vert _{2}+\frac{1}{2}\left\Vert A\cdot\hbar\nabla\varphi\right\Vert _{2}\right)^{2}
\]

\[
\leq C\left(\left\Vert \mathscr{H}_{0}\varphi\right\Vert _{2}^{2}+\left\Vert \varphi\right\Vert _{2}^{2}\right).
\]
where both of the last inequalities follow from Lemma \ref{magnetic maximal inequality }
and Lemma \ref{weighted gradient estimate }. The lower bound for
$\left\Vert \mathscr{O}\varphi\right\Vert _{2}^{2}$ is proved in
a similar manner: 
\[
\left\Vert \mathscr{H}_{0}\varphi\right\Vert _{2}^{2}=\left\Vert \left(\frac{1}{2}\Pi_{k}^{2}+\frac{1}{2}\left|y\right|^{2}\right)\varphi\right\Vert _{2}^{2}
\]

\[
\leq\left\Vert \left(-\hbar^{2}\Delta+\left|y\right|^{2}\right)\varphi\right\Vert _{2}^{2}+\left\Vert A_{k}^{2}(y)\varphi\right\Vert _{2}^{2}+\left\Vert (-i\hbar\partial_{y_{k}}A^{k})\varphi\right\Vert _{2}^{2}+\left\Vert A\cdot i\hbar\nabla\varphi\right\Vert _{2}^{2}
\]

\[
\leq\left\Vert \left(-\hbar^{2}\Delta+\left|y\right|^{2}\right)\varphi\right\Vert _{2}^{2}+K^{2}\left\Vert \left|y\right|^{2}\varphi\right\Vert _{2}^{2}+\hbar^{2}K^{2}\left\Vert \varphi\right\Vert _{2}^{2}+3\underset{\mathbb{R}^{3}}{\int}\hbar^{2}\left|x\right|^{2}\left|\nabla\varphi\right|^{2}(x)dx
\]

\[
\leq\left\Vert \left(-\hbar^{2}\Delta+\left|y\right|^{2}\right)\varphi\right\Vert _{2}^{2}+C\left(\left\Vert \left(-\hbar^{2}\Delta+\left|y\right|^{2}\right)\varphi\right\Vert _{2}^{2}+\left\Vert \varphi\right\Vert _{2}^{2}\right)
\]

\[
+\hbar^{2}K^{2}\left\Vert \varphi\right\Vert _{2}^{2}+C\left(\left\Vert \left(-\hbar^{2}\Delta+\left|y\right|^{2}\right)\varphi\right\Vert _{2}^{2}+\left\Vert \varphi\right\Vert _{2}^{2}\right).
\]

\begin{flushright}
$\square$
\par\end{flushright}

\medskip{}

Finally, in light of the above discussion, the following conclusion
is immediate 
\begin{cor}
The following inclusions hold 
\[
D\left(\mathscr{O}\right)\subset D\left(c_{\hbar}^{\lambda}(x,\xi)\right)
\]

and 
\end{cor}
\[
D\left(\mathscr{H}_{0}\right)\subset D\left(\widetilde{c}_{\hbar}^{\lambda}(x,\xi)\right).
\]

\begin{rem}
\label{remark 7.7} According to Remark 2.11 in {[}17{]} the case
$d=2$ does not introduce any particular difficulties, and the result
remains true up to some minor modifications. However, we were are
unable to locate a reference which includes a full treatment for the
case $d=2$, and therefore we decided to formulate Theorem \ref{Main theorem 2}
in dimension $\geq3$, although likely this can be avoided with some
more effort. Anyhow, the case of $d=2$ with constant magnetic field
of the type $(\mathbb{A})$-- which is the case of interest for Section
\ref{sec:3 OF CHAP 3}-- has been already handled in Lemma 6.6 in
{[}6{]}. 
\end{rem}

\end{document}